\documentclass[a4paper,oneside,10pt]{article}%
\usepackage{amsmath}
\usepackage{xcolor}
\usepackage{amsfonts}
\usepackage{amssymb}
\usepackage{graphicx}
\usepackage[square,numbers,sort&compress]{natbib}%
\providecommand{\U}[1]{\protect\rule{.1in}{.1in}}
\newtheorem{theorem}{Theorem}[section]

\newtheorem{corollary}[theorem]{Corollary}

\newtheorem{definition}[theorem]{Definition}
\newtheorem{assumption}[theorem]{Assumption}

\newtheorem{lemma}[theorem]{Lemma}

\newtheorem{proposition}[theorem]{Proposition}
\newtheorem{remark}[theorem]{Remark}

\newenvironment{proof}[1][Proof]{\noindent\textbf{#1.} }{\ \rule{0.5em}{0.5em}}
\numberwithin{equation}{section}

\begin{document}

\title{Kalman-Bucy filtering and minimum mean square estimator under uncertainty}
\author{Shaolin Ji
\and Chuiliu Kong
\and Chuanfeng Sun
\and Ji-Feng Zhang}

\date{}
\maketitle

\textbf{Abstract}. In this paper, we study a generalized Kalman-Bucy filtering
problem under uncertainty. The drift uncertainty for both signal process and
observation process is considered and the attitude to uncertainty is
characterized by a convex operator (convex risk measure). The optimal filter
or the minimum mean square estimator (MMSE) is calculated by solving the
minimum mean square estimation problem under a convex operator. In the first
part of this paper, this estimation problem is studied under $g$-expectation
which is a special convex operator. For this case, we prove that there exists
a worst-case prior $P^{\theta^{\ast}}$. Based on this $P^{\theta^{\ast}}$ we
obtained the Kalman-Bucy filtering equation under $g$-expectation. In the
second part of this paper, we study the minimum mean square estimation problem
under general convex operators. The existence and uniqueness results of the
MMSE are deduced.

{\textbf{Key words}.} Kalman-Bucy filtering; minimum mean square estimator;
drift uncertainty; convex operator; minimax theorem; backward stochastic
differential equation

\textbf{AMS subject classifications.} 62M20, 60G35, 93E11, 62F86

\section{Introduction}

It is well-known that Kalman-Bucy filtering is the foundation of modern
filtering theory (see Bensoussan\cite{Bensoussan}, Bian and Crisan \cite{BC}, Liptser and Shiryaev \cite{Liptser}, Xiong \cite{Xiong}). It lays the
groundwork for further study of optimization problems under partial
information in various fields. For example, Duncan and Pasik-Dunan \cite{Duncan1}, {Huang, Wang and Zhang \cite{Huang-Wang-Zhang}}, {\O }ksendal and Sulem \cite{Oksendal-Sulem}, Tang \cite{Tang} studied the
optimal control (game) for partially observed stochastic systems; Lakner
\cite{L}, Bensoussan and Keppo \cite{Bensoussan-Keppo} considered the utility
maximization problem under partial information in mathematical finance and so on.

Let's first recall the classic Kalman-Bucy filtering theory. The model is
described as follows: under the probability measure $\mathbb{P}$,%
\begin{equation}
\left\{
\begin{array}
[c]{rl}%
dx(t) & =(B(t)x(t)+b(t))dt+dw(t),\\
x(0) & =x_{0},\\
dm(t) & =(H(t){x}(t)+h(t))dt+dv(t),\\
m(0) & =0
\end{array}
\right.  \label{eq-intro-system-1}%
\end{equation}
where $x(\cdot)$ is the signal process, $m(\cdot)$ is the observation process,
$w(\cdot)$ and $v(\cdot)$ are two independent Brownian motions. The
coefficients $B(t),\ H(t),\ b(t),\ h(t)$ are deterministic uniformly bounded
functions in $t\in\lbrack0,T]$, $x_{0}$ is a given constant vector. Set
$\mathcal{Z}_{t}=\sigma\{m(s);0\leq s\leq t\}$ which represents all the
observable information up to time $t$. The Kalman filter $\bar{x}(t)$ of
$x(t)$ is
\[
\bar{x}(t)=\mathbb{E}_{\mathbb{P}}[x(t)|\mathcal{Z}_{t}]
\]
where $\mathbb{E}_{\mathbb{P}}[\cdot]$ denotes the expectation with respect to
the probability measure $\mathbb{P}$. It is well-known that the optimal
estimator $\bar{x}(t)$ of the signal $x(t)$ solves the following minimum mean
square estimation problem:%
\[
\min_{\zeta\in L_{\mathcal{Z}_{t}}^{2}(\Omega,P)}E_{P}\Vert x(t)-\zeta
\Vert^{2}.
\]
So $\bar{x}(t)$ is also called the minimum mean square estimator, or MMSE for short.

In this paper, we suppose that there exists model uncertainty for the system
(\ref{eq-intro-system-1}). In other words, we don't know the true probability
$\mathbb{P}$ and only know that it falls in a set of probability measures
$\mathcal{P}$ which is called the prior set. For continuous-time models, Chen
and Epstein \cite{Chen-Epstein} first proposed one kind of model uncertainty
which is usually called drift uncertainty. Later Epstein and Ji proposed more
general uncertainty models (see \cite{Epstein-Ji-1} and \cite{Epstein-Ji-2}
for details), {Guo \cite{Guo} introduced some basic scientific problems concerning the estimation, control, and games of dynamical systems with uncertainty and shared some related theoretical progress}. In this paper, we introduce the following drift uncertainty
model: for every $P^{\theta}\in\mathcal{P}$, consider%
\begin{equation}
\left\{
\begin{array}
[c]{rl}%
dx(t) & =(B(t)x(t)+b(t)-\theta_{1}(t))dt+dw^{\theta_{1}}(t),\\
x(0) & =x_{0},\\
dm(t) & =(H(t){x}(t)+h(t)-\theta_{2}(t))dt+dv^{\theta_{2}}(t),\\
m(0) & =0,
\end{array}
\right.  \label{intro-state equ}%
\end{equation}
where $w^{\theta_{1}}$ and $v^{\theta_{2}}$ are Brownian motions under
$P^{\theta}$ and $\theta=(\theta_{1},\theta_{2})\in\Theta$ is called the
uncertainty parameter. When $\theta$ changes, the distribution of the
solutions $x(\cdot)$ and $m(\cdot)$ of the above equations also change. The
question now is how to calculate the Kalman filter in such an uncertain
environment. A natural idea is to calculate the worst-case minimum mean square
estimation problem:%
\begin{equation}
\min_{\zeta}\sup_{P^{\theta}\in\mathcal{P}}E_{P^{\theta}}(\Vert x(t)-\zeta
\Vert^{2}) \label{intro-problem}%
\end{equation}
which is to minimize the maximum expected loss over a range of possible
models. Recently, Borisov \cite{Borisov1} and \cite{Borisov2} studied this
type of estimator for finite state Markov processes with uncertainty of the
transition intensity and the observation matrices. Allan and Cohen
\cite{Allan-Cohen} investigated the Kalman-Bucy filtering with a uncertainty
parameter by a control approach. {Moreover, in the past decade, much research has been
discussed depending on the technique of $H_{\infty}$ filter, see \cite{Chen-Shen}-\cite{Che-Yang} and so on. Different from this paper, the design goal of $H_{\infty}$ filter is to guarantee that the filtering error system is asymptotically stable, while achieving a prescribed $H_{\infty}$ performance level}. From another perspective,
(\ref{intro-problem}) can be rewritten as a minimum mean square estimation
problem under a sublinear operator:%
\[
\min_{\zeta}\mathcal{E(}\Vert x(t)-\zeta\Vert^{2})
\]
where $\mathcal{E}(\cdot):=\sup_{P^{\theta}\in\mathcal{P}}E_{P^{\theta}}[\cdot]$ is a
sublinear operator. Recently, Ji, Kong and Sun \cite{Ji-Kong-Sun} and
\cite{Ji-Kong-Sun-amtomatica} studied Kalman-Bucy filtering under sublinear
operators when the drift uncertainty appears in the signal process and the
observation process respectively. The related literatures about the minimum
mean square estimation problems under sublinear operators include Sun, Ji
\cite{JS} and Ji, Kong, Sun \cite{Ji-Kong-Sun-1} in which they considered
these problems on $L^{\infty}(\Omega,P)$ and $L^{p}(\Omega,P)$ respectively.

However, when we study some problems, especially financial and risk management
problems, we need to use a more general nonlinear operator: the convex
operator or convex risk measure. For example, in the last decade, the concept
of convex risk measure (a special convex operator) has been extensively
studied in various fields (see F\"{o}llmer, Schied \cite{FS}, Arai, Fukasawa
\cite{Arai-Fukasawa} et al). So it is an interesting problem to solve the
minimum mean square estimation problem under the convex operator. Unlike
sublinear operators, the lack of positive homogeneity results in an extra
penalty term in the expression of convex operators. For the convex operator
$\rho(\cdot)$, that is to say, $\rho(\cdot)$ can be represented as
\[
\rho(\cdot)=\sup_{P^{\theta}\in\mathcal{P}}[E_{P^{\theta}}[\cdot
]-\alpha(P^{\theta})],
\]
where $\alpha(P^{\theta})$ is a penalty function defined on a probability
measure set. If $\rho(\cdot)$ is sublinear, the $\alpha(P^{\theta})$ takes
values in $\{0,\infty\}$. The main difference between this paper and the
previous ones is how to deal with the penalty term.

In this paper, we first generalize the Kalman-Bucy filtering to accommodate
drift uncertainty in both signal process and observation process and the
attitude to uncertainty is characterized by a convex operator (convex risk
measure). In more details, we consider system (\ref{intro-state equ}) and
calculate the MMSE by solving%
\[
\underset{\zeta}{\min}\sup_{P^{\theta}}[E_{P^{\theta}}[\Vert x(t)-\zeta
\Vert^{2}]+\alpha_{0,t}(P^{\theta})]=\underset{\zeta}{\min}\mathcal{E}%
_{g}[\Vert x(t)-\zeta\Vert^{2}]
\]
where
\begin{equation}
\mathcal{E}_{g}[\cdot]:=\sup_{P^{\theta}}[E_{P^{\theta}}[\Vert\cdot\Vert
^{2}]+\alpha_{0,t}(P^{\theta})] \label{g-expectation}%
\end{equation}
is called $g$-expectation introduced by Peng \cite{Peng-1}. In our context,
$\mathcal{E}_{g}[\cdot]$ is a special convex operator and (\ref{g-expectation}%
) is it's dual representation obtained in El Karoui et al \cite{EPQ}. Under
some mild conditions, we prove that there exists a worst-case prior
$P^{\theta^{\ast}}$. Based on this $P^{\theta^{\ast}}$ we obtained the
filtering equation by which the MMSE $\hat{x}$ is governed.

The convex $g$-expectation is just a special convex operator. It is worth
studying the minimum mean square estimation problem under the general convex
operator. In the second part of this paper, we solve the following problem
(For the convenience of readers, we misused some notations in the introduction
and Section \ref{section-general problem}):%
\[
\min_{\zeta}\rho(\Vert{x}(t)-\zeta\Vert^{2})
\]
where $\rho(\cdot)$ is a general convex operator (convex risk measure). The
existence and uniqueness results of the MMSE under the general convex operator
are deduced.

The paper is organized as follows. In Section 2, we give some preliminaries
and formulate our filtering problem under $g$-expectations. In Section 3, the
worst-case prior $P^{\theta^{\ast}}$ is obtained and the corresponding
Kalman-Bucy filtering equation (\ref{convex theta optimal solution}) is
deduced. We study the minimum mean square estimation problem under general
convex operators $L_{\mathcal{F}}^{p}(\mathbb{P})$ and obtain the existence
and uniqueness results of the MMSE in Section 4.

\section{Preliminaries and problem formulation}
%Here are some sub-sections:
Let $(\Omega,\mathcal{F},\mathbb{P})$ be a complete probability space on which
two independent $n$-dimensional and $m$-dimensional Brownian motions
$w(\cdot)$ and $v(\cdot)$ are defined. For the sake of generality, they are
not standard. The means of $w(\cdot)$ and $v(\cdot)$ are zero and the
covariance matrices are $Q(\cdot)$ and $R(\cdot)$ respectively. We assume that
the matrix $R(\cdot)$ is uniformly positive definite. For a fixed time $T>0$,
denote by $\mathbb{F=}\{\mathcal{F}_{t},0\leq t\leq T\}$ the natural
filtration of $w(\cdot)$ and $v(\cdot)$ satisfying the usual conditions. We
assume $\mathcal{F}={\mathcal{F}_{T}}$. For any given Euclidean space
$\mathbb{H}$, denote by $\langle\cdot,\cdot\rangle$ (resp. $\Vert\cdot\Vert$)
the scalar product (resp. norm) of $\mathbb{H}$. Let $A^{\intercal}$ denote
the transpose of a matrix $A$. For a $\mathbb{R}^{n}$-valued vector
$x=(x_{1},\cdot\cdot\cdot,x_{n})^{\intercal}$, $|x|:=(|x_{1}|,\cdot\cdot
\cdot,|x_{n}|)^{\intercal}$; for two $\mathbb{R}^{n}$-valued vectors $x$ and
$y$, $x\leq y$ means that $x_{i}\leq y_{i}$ for $i=1,\cdot\cdot\cdot,n$.
Through out this paper, $0$ denotes the matrix/vector with appropriate
dimension whose all entries are zero. For $1<p<\infty$, denote by
$L_{\mathbb{F}}^{p}(0,T;\mathbb{H})$ the space of all the $\mathbb{F}$-adapted
$\mathbb{H}$-valued stochastic processes on $[0,T]$ such that
\[
\mathbb{E}\left[  \int_{0}^{T}\Vert f(r)\Vert^{p}dr\right]  <\infty,\;\forall
f\in L_{\mathbb{F}}^{p}(0,T;\mathbb{H}).
\]

The Kalman-Bucy filtering theory is based on a reference probability measure
$\mathbb{P}$ for the system \eqref{eq-intro-system-1}. However, if we don't
know the true probability measure $\mathbb{P}$ and only know that it falls in
the set $\mathcal{P}$ which is a suitably chosen space of equivalent
probability measures, then it is naturally to study the worst-case minimum
mean square estimators (MMSE).

\subsection{Prior set and $g$-expectation}
In order to characterize uncertainty, we introduce the prior set $\mathcal{P}$
and $g$-expectation which is a special convex operator.

Let $\theta(\cdot)=(\theta_{1}(\cdot),\theta_{2}(\cdot))^{\intercal}$\ be a
$\mathbb{R}^{n+m}$-valued progressively measurable process on $[0,T]$. For a
given constant $\mu$, let $\Theta$ be the set of all $\mathbb{R}^{n+m}$-valued
progressively measurable processes $\theta$ with $|\theta_{i}(t)|\leq
\mu,\ 0\leq t\leq T$. Define
\begin{equation}
\mathcal{P}=\{P^{\theta}|\frac{dP^{\theta}}{d\mathbb{P}}=f^{\theta
}(T)\ \mbox{for}\ \theta\in\Theta\} \label{covex probability set}%
\end{equation}
where
\[
f^{\theta}(T):=\exp\big(-\int_{0}^{T}\theta_{1}^{\intercal}(t)dw(t)-\frac
{1}{2}\int_{0}^{T}\Vert\theta_{1}(t)\Vert^{2}dt-\int_{0}^{T}\theta
_{2}^{\intercal}(t)dv(t)-\frac{1}{2}\int_{0}^{T}\Vert\theta_{2}(t)\Vert
^{2}dt\big).
\]
Due to the boundedness of $\theta$, the Novikov's condition holds (see
Karatzas, Shreve \cite{K-S}). Therefore, $P^{\theta}$ defined by
\eqref{covex probability set} is a probability measure which is equivalent to
probability measure $\mathbb{P}$ and the processes $w^{\theta_{1}%
}(t)=w(t)+\int_{0}^{t}\theta_{1}(s)ds$ and $v^{\theta_{2}}(t)=v(t)+\int%
_{0}^{t}\theta_{2}(s)ds$ are Brownian motions under this probability measure
$P^{\theta}$ by Girsanov's theorem. The set $\Theta$ characterizes the
ambiguity and $\mathcal{P}$ is usually called the prior set.

Then, we introduce $g$-expectation and it's dual representation (see
\cite{Peng-1} and \cite{EPQ}). In the following we will see that
$g$-expectation is a powerful tool for studying uncertainty.

\begin{definition}
\label{standard bsde generator} we call a function $g:\Omega\times
\lbrack0,T]\times\mathbb{R}^{n}\times\mathbb{R}^{m}\rightarrow\mathbb{R}$ a
standard generator if it satisfies the following conditions:

\begin{itemize}
\item $(g(\omega,t,z_{1},z_{2}))_{t\in\lbrack0,T]}$ is an adapted process with
\[
\mathbb{E}\int_{0}^{T}|g(\omega,t,z_{1},z_{2})|^{2}dt<\infty
\]
for all $z_{1}\in\mathbb{R}^{n}$ and $z_{2}\in\mathbb{R}^{m}$;
\item $g(\omega,t,z_{1},z_{2})$ is Lipschitz continuous in $z_{1}$ and $z_{2}$,
uniformly in $t$ and $\omega$: there exists constant $\mu>0$ such that for all
$z_{1},\tilde{z}_{1}\in\mathbb{R}^{n}$ and $z_{2},\tilde{z}_{2}\in
\mathbb{R}^{m}$ we have
\[
|g(\omega,t,z_{1},z_{2})-g(\omega,t,\tilde{z}_{1},\tilde{z}_{2})|\leq \mu(\Vert
z_{1}-\tilde{z}_{1}\Vert+\Vert{z}_{2}-\tilde{z}_{2}\Vert);
\]
\item $g(\omega,t,0,0)=0$ for all $t\geq0$ and $\omega\in\Omega$.
\end{itemize}
\end{definition}

For a standard generator $g$, the following backward stochastic differential
equation (BSDE for short)
\[
\left\{
\begin{array}
[c]{rl}%
-dY(t) & =g(t,Z_{1}(t),Z_{2}(t))dt-Z_{1}^{\intercal}(t)dw(t)-Z_{2}^{\intercal
}(t)dv(t),\;t\in\lbrack0,T]\\
Y(T) & =\xi
\end{array}
\right.
\]
with terminal condition $\xi\in L_{\mathcal{F}_{T}}^{2}(\Omega,\mathbb{P})$
has a unique square integrable solution $(Y(t),\\ Z_{1}(t), Z_{2}(t))_{t\in
(0,T]}$ (see \cite{Peng-1}). Peng \cite{Peng-1} calls $Y(t):=\mathcal{E}%
_{g}(\xi|\mathcal{F}_{t})$ the (condition) $g$-expectation of $\xi$ at time
$t$.

\begin{definition}
\label{convex bsde generator} A standard generator $g$ is called a convex
generator if $g(\omega,t,z_{1},z_{2})$ is convex in $z_{1}$ and $z_{2}$ for
$z_{1}\in\mathbb{R}^{n}$ and $z_{2}\in\mathbb{R}^{m}$. The $g$-expectation
with a convex generator is called the convex $g$-expectation.
\end{definition}

Now we give the dual representation of the convex $g$-expectation through the
prior set and the concave dual function of $g$.

Let
\begin{equation*}
\begin{array}
[c]{rl}%
&G(\omega,t,\theta_{1},\theta_{2})=\inf_{z_{1}\in\mathbb{R}^{n},z_{2}%
\in\mathbb{R}^{m}}[g(\omega,t,z_{1},z_{2})+\langle z_{1},\theta_{1}%
\rangle+\langle z_{2},\theta_{2}\rangle],\\
&(\omega\in\Omega,t\in
\lbrack0,T],\theta_{1}\in\mathbb{R}^{n},\theta_{2}\in\mathbb{R}^{m})
\end{array}
\end{equation*}
be the concave dual function of $g(\omega,t,z_{1},z_{2})$.

EI Karoui et al. \cite{EPQ} (also see Delbaen et al. \cite{Delbaen2010})
established the following dual representation for $g$-expectation: for a
$\mathcal{F}_{s}$-measurable random variable $\xi$, the $g$-expectation at
time $t$ can be represented as%
\begin{equation}
\mathcal{E}_{g}(\xi|\mathcal{F}_{t})=\sup_{P^{\theta}\in\mathcal{P}%
}[E_{P^{\theta}}[\xi|\mathcal{F}_{t}]+\alpha_{t,s}(P^{\theta})]
\label{convex operator representation}%
\end{equation}
where
\begin{equation}
\alpha_{t,s}(P^{\theta}):=E_{P^{\theta}}[\int_{t}^{s}G(r,\theta_{1}%
(r),\theta_{2}(r))dr|\mathcal{F}_{t}],\;0\leq t\leq s\leq T.
\label{penalty operator}%
\end{equation}

%%%%%%%%%%%%%%%%%%%%%%%%%%%%%%%%%%%%%%%%%%%%%%%%%%%%%%%%%%%%%%%%%%%%%%%%%%%%%%%%%%%%%%%%%%%%%%%%%%%%%%%%%%%%%%%%%%%%%%%%%%%%%%%%%%%%%%%%%%%%%%%%%%%%%%%%%%%%%%%%%%
\begin{remark}
{It is easy to check that $\mathcal{E}_{g}(\cdot|\mathcal{F}_{t})$ is a special
convex operator (see (\ref{def2.1})). Moreover, if we let the standard generator $g(t,z_1,z_2)=\mu(|z_1|+|z_2|)$, then the corresponding dual function of $g(t,z_1,z_2)$ and penalty term $\alpha
_{t,s}(P^{\theta})$ are simultaneously equal to $0$. Then the above convex operator $\mathcal{E}_{g}(\cdot|\mathcal{F}_{t})$ degenerates to a sublinear operator.}
\end{remark}

\subsection{Problem formulation \label{problem formulation--kalman-bucy}}

We formulate the Kalman-Bucy filtering problem under uncertainty. For every
$\theta\in\Theta$, under the probability measure $P^{\theta}\in\mathcal{P}$
\begin{equation}
\left\{
\begin{array}
[c]{rl}%
dx(t) & =(B(t)x(t)+b(t)-\theta_{1}(t))dt+dw^{\theta_{1}}(t),\\
x(0) & =x_{0},\\
dm(t) & =(H(t){x}(t)+h(t)-\theta_{2}(t))dt+dv^{\theta_{2}}(t),\\
m(0) & =0,
\end{array}
\right.  \label{convex generalize K-B}%
\end{equation}
where $x(\cdot)\in L_{\mathbb{F}}^{2}(0,T;\mathbb{R}^{n})$ is the signal
process and ${m}(\cdot)\in L_{\mathbb{F}}^{2}(0,T;\mathbb{R}^{m})$ is the
observation process. The coefficients $B(t)\in\mathbb{R}^{n\times n}%
,\ H(t)\in\mathbb{R}^{m\times n},\ b(t)\in\mathbb{R}^{n},\ h(t)\in
\mathbb{R}^{m}$ are deterministic uniformly bounded functions in $t\in
\lbrack0,T]$, $x_{0}\in\mathbb{R}^{n}$ is a given constant vector. Set
\[
\mathcal{Z}_{t}=\sigma\{m(s);0\leq s\leq t\}
\]
which represents all the observable information up to time $t$. We want to
calculate the MMSE of the signal $x(t)$ by solving the following worst-case
minimum mean square estimation problem:%
\begin{equation}
\begin{array}
[c]{rl}%
\inf\limits_{\zeta(t)\in L_{\mathcal{Z}_{t}}^{2+\epsilon}(\Omega,\mathbb{P}%
,\mathbb{R}^{n})}\mathcal{E}_{g}(\Vert{x}(t)-\zeta(t)\Vert^{2})=&\inf\limits
_{\zeta(t)\in L_{\mathcal{Z}_{t}}^{2+\epsilon}(\Omega,\mathbb{P}%
,\mathbb{R}^{n})}\sup\limits_{P^{\theta}\in\mathcal{P}}[E_{P^{\theta}}(\Vert
{x}(t)-\zeta(t)\Vert^{2})\\
&+\alpha_{0,t}(P^{\theta})]
\end{array}
\label{convex robust problem}%
\end{equation}
where $L_{\mathcal{Z}_{t}}^{2+\epsilon}(\Omega,\mathbb{P},\mathbb{R}^{n})$ is
the set of all the $\mathbb{R}^{n}$-valued $(2+\epsilon)$ integrable
$\mathcal{Z}_{t}$-measurable random variables and $0<\epsilon<1$.

\begin{definition}
If $\hat{x}(t)\in L_{\mathcal{Z}_{t}}^{2+\epsilon}(\Omega,\mathbb{P}%
,\mathbb{R}^{n})$ satisfies%
\[
\mathcal{E}_{g}(\Vert{x}(t)-\hat{x}(t)\Vert^{2})=\inf_{\zeta(t)\in
L_{\mathcal{Z}_{t}}^{2+\epsilon}(\Omega,\mathbb{P},\mathbb{R}^{n})}%
\mathcal{E}_{g}(\Vert{x}(t)-\zeta(t)\Vert^{2}),
\]
then we call $\hat{x}(t)$ the minimum mean square estimator (MMSE) of ${x}(t)$.
\end{definition}

\section{Kalman-Bucy filtering under $g$-expectation\label{Kalman-Bucy filtering under uncertainty}}

In this section, we calculate the minimum mean square estimator $\hat{{x}}(t)$
of the problem \eqref{convex robust problem} for $t\in\lbrack0,T]$. Without
loss of generality, all the statements in this section are only proved in the
one dimensional case.

\begin{lemma}
\label{convex compact} The set $\{\frac{dP^{\theta}}{d\mathbb{P}}:P^{\theta
}\in\mathcal{P}\}\subset L^{1+\frac{2}{\epsilon}}(\Omega,\mathcal{F}%
,\mathbb{P})$ is $\sigma(L^{1+\frac{2}{\epsilon}}(\Omega,\mathcal{F}%
,\mathbb{P}), L^{1+\frac{\epsilon}{2}}(\Omega,\mathcal{F}, \mathbb{P}))$-compact
and $\mathcal{P}$ is convex.
\end{lemma}

\begin{proof}
Since $\theta$ is bounded, by Theorem \ref{convex AP} in the Appendix, the set
$\{\frac{dP^{\theta}}{d\mathbb{P}}:P^{\theta}\in\mathcal{P}\}$ is bounded in
norm $\|\cdot\|_{1+\frac{2}{\epsilon}}$. From Theorem $4.1$ of Chapter $1$ in
Simons \cite{Simons}, we know that the set $\{\frac{dP^{\theta}}{d\mathbb{P}%
}:P^{\theta}\in\mathcal{P}\}$ is $\sigma(L^{1+\frac{2}{\epsilon}}%
(\Omega,\mathcal{F},\mathbb{P}),L^{1+\frac{\epsilon}{2}}(\Omega,\mathcal{F}%
,\mathbb{P}))$-compact.

Let $\theta^{1}=(\theta_{1}^{1},\theta_{2}^{1})^{\intercal}$ and $\theta
^{2}=(\theta_{1}^{2},\theta_{2}^{2})^{\intercal}$ belong to $\Theta$.
$f^{\theta^{1}}$ and $f^{\theta^{2}}$ denote the corresponding exponential
martingales: for $t\in\lbrack0,T]$,
\[
f^{\theta^{i}}(t)=\exp(\int_{0}^{t}\theta_{1}^{i}(s)dw(s)-\frac{1}{2}\int%
_{0}^{t}(\theta_{1}^{i}(s))^{2}ds+\int_{0}^{t}\theta_{2}^{i}(s)dv(s)-\frac
{1}{2}\int_{0}^{t}(\theta_{2}^{i}(s))^{2}ds)
\]
which satisfies
\[
df^{\theta^{i}}(t)=f^{\theta^{i}}(t)(\theta_{1}^{i}(t)dw(t)+\theta_{2}%
^{i}(t)dv(t)),\ i=1,2.
\]
Let $\lambda_{1}$ and $\lambda_{2}$ be nonnegative constants which belong to
$(0,1)$ with $\lambda_{1}+\lambda_{2}=1$. Define
\[
\left\{
\begin{array}
[c]{rl}%
\theta_{1}^{\lambda}(t) & =\frac{\lambda_{1}\theta_{1}^{1}(t)f^{\theta^{1}%
}(t)+\lambda_{2}\theta_{1}^{2}(t)f^{\theta^{2}}(t)}{\lambda_{1}f^{\theta^{1}%
}(t)+\lambda_{2}f^{\theta^{2}}(t)},\\
\theta_{2}^{\lambda}(t) & =\frac{\lambda_{1}\theta_{2}^{1}(t)f^{\theta^{1}%
}(t)+\lambda_{2}\theta_{2}^{2}(t)f^{\theta^{2}}(t)}{\lambda_{1}f^{\theta^{1}%
}(t)+\lambda_{2}f^{\theta^{2}}(t)}.
\end{array}
\right.
\]
It is easy to verify that
\[
d(\lambda_{1}f^{\theta^{1}}(t)+\lambda_{2}f^{\theta^{2}}(t))=(\lambda
_{1}f^{\theta^{1}}(t)+\lambda_{2}f^{\theta^{2}}(t))(\theta_{1}^{\lambda
}(t)dw(t)+\theta_{2}^{\lambda}(t)dv(t)).
\]
Since $f^{\theta^{i}}(t)>0,\ i=1,2$, the process $\theta^{\lambda}=(\theta
_{1}^{\lambda},\theta_{2}^{\lambda})^{\intercal}$ belongs to $\Theta$.
Therefore, it results in that $\mathcal{P}$ is convex. This completes the proof.
\end{proof}

\begin{lemma}
\label{convex penalty operator} The penalty term $\alpha_{0,T}(P^{\theta})$ is
a concave functional on $\mathcal{P}$.
\end{lemma}

\begin{proof}
Let $\theta^{1}=(\theta_{1}^{1},\theta_{2}^{1})^{\intercal}$ and $\theta
^{2}=(\theta_{1}^{2},\theta_{2}^{2})^{\intercal}$ belong to $\Theta$.
$f^{\theta^{1}}$ and $f^{\theta^{2}}$ denote the exponential martingales
respectively as in Lemma \ref{convex compact}. By Lemma \ref{convex compact},
the exponential martingale $(\lambda_{1}\frac{dP^{\theta^{1}}}{d\mathbb{P}%
}+\lambda_{2}\frac{dP^{\theta^{2}}}{d\mathbb{P}})$ is generated by
$\theta^{\lambda}=(\theta_{1}^{\lambda},\theta_{2}^{\lambda})$. It yields that%
\[
\alpha_{0,T}(\lambda_{1}P^{\theta^{1}}+\lambda_{2}P^{\theta^{2}}%
)=\mathbb{E}[(\lambda_{1}f^{\theta^{1}}(T)+\lambda_{2}f^{\theta^{2}}%
(T))\int_{0}^{T}G(t,\theta_{1}^{\lambda}(t),\theta_{2}^{\lambda}(t))dt].
\]
Since $G(t,\cdot,\cdot)$ is a concave function, we have%
\[%
\begin{array}
[c]{rl}%
&\alpha_{0,T}(\lambda_{1}P^{\theta^{1}}+\lambda_{2}P^{\theta^{2}})\\
&\geq
\mathbb{E}[(\lambda_{1}f^{\theta^{1}}(T)+\lambda_{2}f^{\theta^{2}}%
(T))(\int_{0}^{T}\frac{\lambda_{1}f^{\theta^{1}}(t)}{\lambda_{1}f^{\theta^{1}%
}(t)+\lambda_{2}f^{\theta^{2}}(t)}G(t,\theta_{1}^{1}(t),\theta_{2}^{1}(t))\\
&+\int_{0}^{T}\frac{\lambda_{2}f^{\theta^{2}}(t)}{\lambda_{1}f^{\theta^{1}%
}(t)+\lambda_{2}f^{\theta^{2}}(t)}G(t,\theta_{1}^{2}(t),\theta_{2}%
^{2}(t)))dt]\\
& =\mathbb{E}[(\int_{0}^{T}{\lambda_{1}f^{\theta^{1}}(t)}G(t,\theta_{1}%
^{1}(t),\theta_{2}^{1}(t))+\int_{0}^{T}{\lambda_{2}f^{\theta^{2}}%
(t)}G(t,\theta_{1}^{2}(t),\theta_{2}^{2}(t)))dt]\\
& =\mathbb{E}[{\lambda_{1}f^{\theta^{1}}(T)}\int_{0}^{T}G(t,\theta_{1}%
^{1}(t),\theta_{2}^{1}(t))dt]+\mathbb{E}[{\lambda_{2}f^{\theta^{2}}(T)}%
\int_{0}^{T}G(t,\theta_{1}^{2}(t),\\
&\theta_{2}^{2}(t))dt]=\lambda_{1}\alpha_{0,T}(P^{\theta^{1}})+\lambda_{2}\alpha_{0,T}%
(P^{\theta^{2}}).
\end{array}
\]
Therefore, the penalty term $\alpha(P^{\theta})$ is a concave functional on
$\mathcal{P}$. This completes the proof.
\end{proof}
%%%%%%%%%%%%%%%%%%%%%%%%%%%%%%%%%%%%%%%%%%%%%%%%%%%%%%%%%%%%%%%%%%%%%%%%%%%%%%%%%%%%%%%%%%%%%%%%%%%%%%%%%%%%%%%%%%%%%%%%%%%%%%%%%%%%%%%%%%%%%%%%%%%%%%%%%%%%%%%%%%
\begin{remark}
It is easy to check that for any $t\in\lbrack0,T]$, $\alpha_{0,t}(P^{\theta})$
is a concave functional on $\mathcal{P}$ and%
\[
\alpha_{0,t}(P^{\theta})=\mathbb{E}[f^{\theta}(T)\cdot\int_{0}^{t}%
G(s,\theta_{1}(s),\theta_{2}(s))ds]=\mathbb{E}[f^{\theta}(t)\cdot\int_{0}%
^{t}G(s,\theta_{1}(s),\theta_{2}(s))ds].
\]
\end{remark}

\begin{lemma}\label{theta convergence}
Suppose that the stochastic processes $(g_m(t))_{t\in[0,T]}, m=1,2,...$ and $(f^{*}(t))_{t\in[0,T]}$ are exponential martingales respect to the filtration $\mathbb{F}$ and $(g_{m}(T)-f^{\ast}(T))\xrightarrow{L^{2}(\Omega, \mathcal{F},\mathbb{P})}0$. Then for any $0\leq t\leq T$, we have
\[
(\theta^m_i(t)-\theta^{\ast}_i(t))\xrightarrow{L^{2}(\Omega, \mathcal{F},\mathbb{P})}0,\ i=1,2,
\]
where $\theta^{m}(t)=(\theta^m_1(t),\theta^m_2(t))\in\Theta$ and $\theta^{\ast}(t)=(\theta^{\ast}_1(t),\theta^{\ast}_2(t))\in\Theta$ are respectively generators of $(g_m(t))_{t\in[0,T]}, m=1,2,...$ and $(f^{*}(t))_{t\in[0,T]}$.
\end{lemma}

\begin{proof}
Denote the generator of $g_{m}(\cdot)$ by $\theta^{m}=(\theta_{1}^{m}%
,\theta_{2}^{m})$, i.e., for $0\leq t\leq T$,
\[
g_{m}(t)=\exp(\int_{0}^{t}\theta_{1}^{m}(s)dw(s)-\frac{1}{2}\int_{0}%
^{t}(\theta_{1}^{m}(s))^{2}ds+\int_{0}^{t}\theta_{2}^{m}(s)dv(s)-\frac{1}%
{2}\int_{0}^{t}(\theta_{2}^{m}(s))^{2}ds).
\]
We want to prove that $(\theta^{m})$ converges to $\theta^{\ast}$. Since
$g_{m}(\cdot)$ and $f^{\ast}(\cdot)$ are martingales and $g_{m}%
(T)\xrightarrow{L^{2}(\Omega,\mathcal{F},\mathbb{P})}f^{\ast}(T)$, it is easy
to verify that $g_{m}%
(t)\xrightarrow{L^{2}(\Omega,\mathcal{F},\mathbb{P})}f^{\ast}(t)$ for any
$t\in\lbrack0,T]$. Applying It\^{o}'s formula to $(g_{m}(t)-f^{\ast}(t))^{2}$,
we have%
\[%
\begin{array}
[c]{rl}%
&d(g_{m}(t)-f^{\ast}(t))^{2}\\
&=2(g_{m}(t)-f^{\ast}(t))[(g_{m}(t)\theta_{1}%
^{m}(t)-f^{\ast}(t)\theta_{1}^{\ast}(t))dw(t)+(g_{m}(t)\theta_{2}%
^{m}(t)\\
&-f^{\ast}(t)\theta_{2}^{\ast}(t))dv(t)]
+(g_{m}(t)\theta_{1}^{m}(t)-f^{\ast}(t)\theta_{1}^{\ast}(t))^{2}%
dt+(g_{m}(t)\theta_{2}^{m}(t)\\
&-f^{\ast}(t)\theta_{2}^{\ast}(t))^{2}dt.
\end{array}
\]
Taking expectation on both sides,
\begin{equation}
\begin{array}
[c]{rl}
\mathbb{E}[(g_{m}(T)-f^{\ast}(T))^{2}]=&\mathbb{E}[\int_{0}^{T}(g_{m}%
(t)\theta_{1}^{m}(t)-f^{\ast}(t)\theta_{1}^{\ast}(t))^{2}dt]\\
&+\mathbb{E}%
[\int_{0}^{T}(g_{m}(t)\theta_{2}^{m}(t)-f^{\ast}(t)\theta_{2}^{\ast}%
(t))^{2}].
\end{array}
\end{equation}
Since $\lim\limits_{m\rightarrow\infty}\mathbb{E}[(g_{m}(T)-f^{\ast}%
(T))^{2}]=0$, it yields that
\begin{equation}
\lim\limits_{m\rightarrow\infty}\mathbb{E}[\int_{0}^{T}(g_{m}(t)\theta_{i}%
^{m}(t)-f^{\ast}(t)\theta_{i}^{\ast}(t))^{2}dt]=0,\;i=1,2.
\label{convex-convergence1}%
\end{equation}
Note that
\[%
\begin{array}
[c]{rl}%
&\mathbb{E}[\int_{0}^{T}(g_{m}(t)\theta_{1}^{m}(t)-f^{\ast}(t)\theta_{1}^{\ast
}(t))^{2}dt] \\
&=\mathbb{E}\int_{0}^{T}[(f^{\ast}(t)-g_{m}(t))^{2}(\theta
_{1}^{\ast}(t))^{2}+(g_{m}(t))^{2}(\theta_{1}^{\ast}(t)-\theta_{1}^{m}(t))^{2}\\
&+2(f^{\ast}(t)-g_{m}(t))g_{m}(t)\theta_{1}^{\ast}(t)(\theta_{1}^{\ast
}(t)-\theta_{1}^{m}(t))]dt.
\end{array}
\]
Because $g_{m}(t)\xrightarrow{L^{2}(\Omega,\mathcal{F},\mathbb{P})}f^{\ast
}(t)$ and $\theta$ is bounded, we have
\[
\begin{array}
[c]{rl}%
\lim\limits_{m\rightarrow\infty}\mathbb{E}[(f^{\ast}(t)-g_{m}(t))^{2}%
(\theta_{1}^{\ast}(t))^{2}]  &  =0;\\
\lim\limits_{m\rightarrow\infty}\mathbb{E}[(f^{\ast}(t)-g_{m}(t))g_{m}%
(t)\theta_{1}^{\ast}(t)(\theta_{1}^{\ast}(t)-\theta_{1}^{m}(t))]  &  =0.
\end{array}
\]
Therefore, $\lim\limits_{m\rightarrow\infty}\mathbb{E}[(g_{m}(t))^{2}%
(\theta_{1}^{\ast}(t)-\theta_{1}^{m}(t))^{2}]=0$. It results in that
$(g_{m}(t))^{2}(\theta_{1}^{\ast}(t)-\theta_{1}^{m}(t))^{2}%
\xrightarrow{\mathbb{P}}0$. Since $g_{m}(t)\xrightarrow{\mathbb{P}}f^{\ast
}(t)$, we have $(\theta_{1}^{\ast}(t)-\theta_{1}^{m}(t))^{2}%
\xrightarrow{\mathbb{P}}0$. Due to the boundedness of $\theta$, we obtain
$(\theta_{1}^{\ast}(t)-\theta_{1}^{m}%
(t))\xrightarrow{L^{2}(\Omega,\mathcal{F},\mathbb{P})}0$. Similarly, we can
obtain $(\theta_{2}^{\ast}(t)-\theta_{2}^{m}%
(t))\xrightarrow{L^{2}(\Omega,\mathcal{F},\mathbb{P})}0$. This completes the proof.
\end{proof}

In the following, we prove that the worst-case\ prior $P^{\theta^{\ast}}$ exists.

\begin{theorem}
\label{convex equivalent representation} For a given $t\in\lbrack0,T]$, there
exists a $\theta^{\ast}\in\Theta$ such that
\begin{equation}%
\begin{array}
[c]{rl}%
&\inf\limits_{\zeta(t)\in L_{\mathcal{Z}_{t}}^{2+\epsilon}(\Omega
,\mathbb{P},\mathbb{R})}\mathcal{E}_{g}[({x}(t)-\zeta(t))^{2}]\\
&=\sup\limits_{P^{\theta}\in\mathcal{P}}\inf\limits_{\zeta(t)\in L_{\mathcal{Z}%
_{t}}^{2+\epsilon}(\Omega,\mathbb{P},\mathbb{R})}[E_{P^{\theta}}[({x}%
(t)-\zeta(t))^{2}]+\alpha_{0,t}(P^{\theta})]\\
&=\inf\limits_{\zeta(t)\in L_{\mathcal{Z}_{t}}^{2+\epsilon}(\Omega
,\mathbb{P},\mathbb{R})}[E_{P^{\theta^{\ast}}}[({x}(t)-\zeta(t))^{2}%
]+\alpha_{0,t}(P^{\theta^{\ast}})].
\end{array}
\label{convex equivalent equation}%
\end{equation}

\end{theorem}

\begin{proof}
Firstly, we prove the first equality. According to Lemmas \ref{convex compact} and \ref{convex penalty operator}, the original robust estimation problem \eqref{convex robust problem} satisfies minimax theorem \ref{minmax}. Therefore, the first equality is verified.

Secondly, we prove the second equality. Choose a sequence $\{\theta^{n}\}$, $n=1,2,\cdot\cdot\cdot$ such that%
\begin{equation}\label{sup proposition}
\begin{array}
[c]{rl}
&\lim\limits_{n\rightarrow\infty}\inf\limits_{\zeta(t)\in L_{\mathcal{Z}_{t}}^{2+\epsilon
}(\Omega,\mathbb{P},\mathbb{R})}[E_{P^{\theta^{n}}}[({x}(t)-\zeta
(t))^{2}]+\alpha_{0,t}(P^{\theta^{n}})]\\
&=\sup\limits_{P^{\theta}\in\mathcal{P}}%
\inf\limits_{\zeta(t)\in L_{\mathcal{Z}_{t}}^{2+\epsilon}(\Omega,\mathbb{P}%
,\mathbb{R})}[\alpha_{0,t}(P^{\theta})+E_{P^{\theta}}[({x}(t)-\zeta(t))^{2}]].%
\end{array}
\end{equation}
Set $f^{\theta^{n}}(T)=\frac{dP^{\theta^{n}}}{d\mathbb{P}}$. By Koml\'{o}s
theorem A.3.4 in \cite{Pham}, there exists a subsequence $\{f^{\theta^{n_{k}}%
}(T)\}_{k\geq1}$ of $\{f^{\theta^{n}}(T)\}_{n\geq1}$ and a $f^{\ast}(T)\in
L^{1}(\Omega,\mathcal{F},\mathbb{P})$ such that
\begin{equation}
\lim_{m\rightarrow\infty}\dfrac{1}{m}\sum_{k=1}^{m}f^{\theta^{n_{k}}%
}(T)=f^{\ast}(T),\ \mathbb{P}-a.s.. \label{komlos result}%
\end{equation}
Let $g_{m}(T)=\dfrac{1}{m}\displaystyle\sum_{k=1}^{m}f^{\theta^{n_{k}}}(T)$.
We have $g_{m}(T)\xrightarrow{\mathbb{P}-a.s.}f^{\ast}(T)$. By Theorem
\ref{convex AP} in the Appendix, for any given constant $p>1$ and $m$, we have
$\ \mathbb{E}(g_{m}(T))^{K}\leq M$ \ where $K=(1+\frac{2}{\epsilon})p$ and
$M=\exp({(K^{2}-K)}\mu^{2}T)$. Then, we have $\left\{  |g_{m}(T)|^{1+\frac
{2}{\varepsilon}}:m=1,2,\cdot\cdot\cdot\right\} $ is uniformly integrable.
Therefore, it results in that $\ g_{m}%
(T)\xrightarrow{L^{1+\frac{2}{\epsilon}}(\Omega, \mathcal{F},\mathbb{P})}f^{\ast
}(T)$ and $f^{\ast}(T)\in L^{1+\frac{2}{\epsilon}}(\Omega,\mathcal{F}%
,\mathbb{P})$. According to the convexity and weak compactness of the set
$\{\frac{dP^{\theta}}{d\mathbb{P}}:P^{\theta}\in\mathcal{P}\}$, there exists a
$\theta^{\ast}$ such that $\frac{dP^{\theta^{\ast}}}{d\mathbb{P}}=f^{\ast}(T)$.

Then we prove that the probability measure $P^{\theta^{\ast}}$ with respect to
obtained generator $\theta^{\ast}$ satisfies (\ref{convex equivalent equation}%
). Based on \eqref{sup proposition} and \eqref{komlos result}, we have
\begin{equation}%
\begin{array}
[c]{rl}
& \sup\limits_{P^{\theta}\in\mathcal{P}}\inf\limits_{\zeta(t)\in
L_{\mathcal{Z}_{t}}^{2+\epsilon}(\Omega,\mathbb{P},\mathbb{R})}[E_{P^{\theta}%
}[({x}(t)-\zeta(t))^{2}]+\alpha_{0,t}(P^{\theta})]\\
= & \lim\limits_{n\rightarrow\infty}\inf\limits_{\zeta(t)\in L_{\mathcal{Z}%
_{t}}^{2+\epsilon}(\Omega,\mathbb{P},\mathbb{R})}[\mathbb{E}[f^{P^{\theta_{n}%
}}(T)({x}(t)-\zeta(t))^{2}]+\alpha_{0,t}(P^{\theta_{n}})]\\
= & \lim\limits_{k\rightarrow\infty}\inf\limits_{\zeta(t)\in L_{\mathcal{Z}%
_{t}}^{2+\epsilon}(\Omega,\mathbb{P},\mathbb{R})}[\mathbb{E}[f^{P^{\theta
_{n_{k}}}}(T)({x}(t)-\zeta(t))^{2}]+\alpha_{0,t}(P^{\theta_{n_{k}}})]\\
= & \lim\limits_{m\rightarrow\infty}\frac{1}{m}\displaystyle\sum_{k=1}^{m}%
\inf\limits_{\zeta(t)\in L_{\mathcal{Z}_{t}}^{2+\epsilon}(\Omega
,\mathbb{P},\mathbb{R})}[\mathbb{E}[f^{P^{\theta_{n_{k}}}}(T)({x}%
(t)-\zeta(t))^{2}]+\alpha_{0,t}(P^{\theta_{n_{k}}})]\\
\leq & \liminf\limits_{m\rightarrow\infty}\inf\limits_{\zeta(t)\in
L_{\mathcal{Z}_{t}}^{2+\epsilon}(\Omega,\mathbb{P},\mathbb{R})}\frac{1}%
{m}\displaystyle\sum_{k=1}^{m}[\mathbb{E}[f^{P^{\theta_{n_{k}}}}%
(T)({x}(t)-\zeta(t))^{2}]+\alpha_{0,t}(P^{\theta_{n_{k}}})]\\
\leq & \liminf\limits_{m\rightarrow\infty}\inf\limits_{\zeta(t)\in
L_{\mathcal{Z}_{t}}^{2+\epsilon}(\Omega,\mathbb{P},\mathbb{R})}[\mathbb{E}%
[g_{m}(T)({x}(t)-\zeta(t))^{2}]+\alpha_{0,t}(P^{\theta^{m}})]
\end{array}
\label{leq inequality}%
\end{equation}
where the last inequality is due to the concavity of $\alpha(\cdot)$. By \eqref{leq inequality} and Lemma \ref{theta convergence}, it results in that
\begin{equation}%
\begin{array}
[c]{rl}
& \sup\limits_{P^{\theta}\in\mathcal{P}}\inf\limits_{\zeta(t)\in
L_{\mathcal{Z}_{t}}^{2+\epsilon}(\Omega,\mathbb{P},\mathbb{R})}[E_{P^{\theta}%
}[({x}(t)-\zeta(t))^{2}]+\alpha_{0,t}(P^{\theta})]\\
& \geq\inf\limits_{\zeta(t)\in L_{\mathcal{Z}_{t}}^{2+\epsilon}(\Omega
,\mathbb{P},\mathbb{R})}[E_{P^{\theta^{\ast}}}[({x}(t)-\zeta(t))^{2}%
]+\alpha_{0,t}(P^{\theta^{\ast}})]\\
& =\inf\limits_{\zeta(t)\in L_{\mathcal{Z}_{t}}^{2+\epsilon}(\Omega
,\mathbb{P},\mathbb{R})}[\mathbb{E}[\lim\limits_{m\rightarrow\infty}%
g_{m}(T)({x}(t)-\zeta(t))^{2}]\\
&+\mathbb{E}[f^{\ast}(T)\int_{0}%
^{t}G(r,\theta_{1}^{\ast}(r),\theta_{2}^{\ast}(r))dr]]\\
& =\inf\limits_{\zeta(t)\in L_{\mathcal{Z}_{t}}^{2+\epsilon}(\Omega
,\mathbb{P},\mathbb{R})}[\mathbb{E}[\lim\limits_{m\rightarrow\infty}%
g_{m}(T)({x}(t)-\zeta(t))^{2}]\\
&+\mathbb{E}[\lim\limits_{m\rightarrow\infty
}(g_{m}(T)\int_{0}^{t}G(r,\theta_{1}^{m}(r),\theta_{2}^{m}(r))dr)]]\\
& \geq\limsup\limits_{m\rightarrow\infty}\inf\limits_{\zeta(t)\in
L_{\mathcal{Z}_{t}}^{2+\epsilon}(\Omega,\mathbb{P},\mathbb{R})}[\mathbb{E}%
[g_{m}(T)({x}(t)-\zeta(t))^{2}]\\
&+\mathbb{E}[g_{m}(T)\int_{0}^{t}%
G(r,\theta_{1}^{m}(r),\theta_{2}^{m}(r))dr]]\\
& {\geq\sup\limits_{P^{\theta}\in\mathcal{P}}\inf\limits_{\zeta(t)\in
L_{\mathcal{Z}_{t}}^{2+\epsilon}(\Omega,\mathbb{P},\mathbb{R})}[E_{P^{\theta}%
}[({x}(t)-\zeta(t))^{2}]+\alpha_{0,t}(P^{\theta})]}%
\end{array}
\label{geq inequality}%
\end{equation}
where the second inequality is based on the upper semi-continuous property.
Therefore,
\[
\begin{array}
[c]{rl}
&\sup\limits_{P^{\theta}\in\mathcal{P}}\inf\limits_{\zeta(t)\in L_{\mathcal{Z}_{t}%
}^{2+\epsilon}(\Omega,\mathbb{P},\mathbb{R})}[E_{P^{\theta}}[({x}%
(t)-\zeta(t))^{2}]+\alpha_{0,t}(P^{\theta})]\\
&=\inf\limits_{\zeta(t)\in L_{\mathcal{Z}%
_{t}}^{2+\epsilon}(\Omega,\mathbb{P},\mathbb{R})}[E_{P^{\theta^{\ast}}}%
[({x}(t)-\zeta(t))^{2}]+\alpha_{0,t}(P^{\theta^{\ast}})].
\end{array}
\]

By minimax theorem (Theorem \ref{minmax} in the Appendix), we obtain
\[
\begin{array}
[c]{rl}
&\sup\limits_{P^{\theta}\in\mathcal{P}}\inf\limits_{\zeta(t)\in L_{\mathcal{Z}_{t}%
}^{2+\epsilon}(\Omega,\mathbb{P},\mathbb{R})}[E_{P^{\theta}}[({x}%
(t)-\zeta(t))^{2}]+\alpha_{0,t}(P^{\theta})]\\
&=\inf\limits_{\zeta(t)\in L_{\mathcal{Z}%
_{t}}^{2+\epsilon}(\Omega,\mathbb{P},\mathbb{R})}\sup\limits_{P^{\theta}%
\in\mathcal{P}}[E_{P^{\theta}}[({x}(t)-\zeta(t))^{2}]+\alpha_{0,t}(P^{\theta
})]
\end{array}
\]
which implies that
\begin{align*}
&  \inf\limits_{\zeta(t)\in L_{\mathcal{Z}_{t}}^{2+\epsilon}(\Omega
,\mathbb{P},\mathbb{R})}\mathcal{E}_{g}[({x}(t)-\zeta(t))^{2}]\\
&  =\inf_{\zeta(t)\in L_{\mathcal{Z}_{t}}^{2+\epsilon}(\Omega,\mathbb{P}%
,\mathbb{R})}\sup_{P^{\theta}\in\mathcal{P}}[E_{P^{\theta}}[({x}%
(t)-\zeta(t))^{2}]+\alpha_{0,t}(P^{\theta})]\\
&  =\inf_{\zeta(t)\in L_{\mathcal{Z}_{t}}^{2+\epsilon}(\Omega,\mathbb{P}%
,\mathbb{R})}[E_{P^{\theta^{\ast}}}[({x}(t)-\zeta(t))^{2}]+\alpha
_{0,t}(P^{\theta^{\ast}})].
\end{align*}
This completes the proof.
\end{proof}

For the obtained $\theta^{\ast}(t)=(\theta_{1}^{\ast}(t),\theta_{2}^{\ast
}(t))$ in Theorem \ref{convex equivalent representation}, set $\widehat{\theta
_{i}^{\ast}(t)}=E_{P^{\theta^{\ast}}}[\theta_{i}^{\ast}(t)|\mathcal{Z}_{t}]$,
$i=1,2$.

\begin{theorem}
The MMSE $\hat{x}(t)$ of problem \eqref{convex robust problem} equals
$E_{P^{\theta^{\ast}}}[{x}(t)|\mathcal{Z}_{t}]$ and satisfies the following
equation:%
\begin{equation}
\left\{
\begin{array}
[c]{rl}%
d\hat{x}(t) & =(B(t)\hat{x}(t)+b(t)-\widehat{\theta_{1}^{\ast}(t)}%
)dt+(P(t)H(t)-\widehat{x(t)\theta_{2}^{\ast}(t)}\\
&+\hat{x}(t)\widehat{\theta
_{2}^{\ast}(t)})R(t)^{-1}d\hat{I}(t),\\
\hat{x}(0) & =x_{0},
\end{array}
\right.
\label{convex theta optimal solution}%
\end{equation}
where $\theta^{\ast}$ is obtained in Theorem
\ref{convex equivalent representation}, $\widehat{x(t)\theta_{2}^{\ast}%
(t)}:=E_{P^{\theta^{\ast}}}[x(t)\theta_{2}^{\ast}(t)|\mathcal{Z}_{t}]$ and the
so called innovation process $\hat{I}(t):={m}(t)-\int_{0}^{t}(H(s)\hat
{x}(s)+g(s)-\widehat{\theta_{2}^{\ast}(s)})ds$, $0\leq t\leq T$ is a
$\mathcal{Z}_{t}$-measurable Brownian motion. The variance of the estimation
error $P(t)=E_{P^{\theta^{\ast}}}[(x(t)-\hat{x}(t))^{2}]$ satisfies the
following equation:
\begin{equation}
\left\{
\begin{array}
[c]{rl}%
\frac{dP(t)}{dt}  =&-E_{P^{\theta^{\ast}}}[(P(t)H(t)-\widehat{x(t)\theta
_{2}^{\ast}(t)}+\hat{x}(t)\widehat{\theta_{2}^{\ast}(t)})R^{-1}%
(t)(H(t)P(t)-\widehat{\theta_{2}^{\ast}(t)x(t)}\\
&+\widehat{\theta_{2}^{\ast}%
(t)}\hat{x}(t))]+2E_{P^{\theta^{\ast}}}[-\widehat{x(t)\theta_{1}^{\ast}(t)}+\hat
{x}(t)\widehat{\theta_{1}^{\ast}(t)}]+2B(t)P(t)+Q(t),\\
P(0)  =&0.
\end{array}
\right.  \label{convex Riccati equation}%
\end{equation}

\end{theorem}

\begin{proof}
For the obtained optimal $\theta^{\ast}(t)=(\theta_{1}^{\ast}(t),\theta
_{2}^{\ast}(t))$ in Theorem \ref{convex equivalent representation}, the system
\eqref{convex generalize K-B} and problem \eqref{convex robust problem} can be
reformulated correspondingly under $P^{\theta^{\ast}}$. In more detail, on the
filtered probability space $(\Omega,\mathcal{F},\{\mathcal{F}_{t}\}_{0\leq
t\leq T},P^{\theta^{\ast}})$, the processes $x(\cdot)$ and $m(\cdot)$ satisfy
the following equations:
\begin{equation}
\left\{
\begin{array}
[c]{rl}%
{dx}(t) & =(B(t){x}(t)+b(t)-\theta_{1}^{\ast}(t))dt+dw^{\theta_{1}^{\ast}%
}(t),\\
{x}(0) & =x_{0},\\
{dm}(t) & =(H(t)x(t)+h(t)-\theta_{2}^{\ast}(t))dt+dv^{\theta_{2}^{\ast}}(t),\\
{m}(0) & =0.
\end{array}
\right.  \label{convex theta model}%
\end{equation}
We solve the minimum mean square estimation problem
\begin{equation}
\inf_{\zeta(t)\in L_{\mathcal{Z}_{t}}^{2+\epsilon}(\Omega,\mathbb{P}%
,\mathbb{R})}[E_{P^{\theta^{\ast}}}[({x}(t)-\zeta(t))^{2}]+\alpha
_{0,t}(P^{\theta^{\ast}})]. \label{convex theta problem}%
\end{equation}
Since $\alpha_{0,t}(P^{\theta^{\ast}})$ is a constant, we only need to
consider the following optimization problem:%
\begin{equation}
\inf_{\zeta(t)\in L_{\mathcal{Z}_{t}}^{2+\epsilon}(\Omega,\mathbb{P}%
,\mathbb{R})}[E_{P^{\theta^{\ast}}}[({x}(t)-\zeta(t))^{2}].
\label{convex auxiliary problem}%
\end{equation}

In \cite{Liptser}, Liptser and Shiryaev studied the optimal estimator of the
following problem:%
\begin{equation}
\inf_{\zeta(t)\in L_{\mathcal{Z}_{t}}^{2}(\Omega,{P}^{\theta^{\ast}%
},\mathbb{R})}E_{P^{\theta^{\ast}}}[({x}(t)-\zeta(t))^{2}].
\label{convex square problem}%
\end{equation}
By Theorem 8.1 in \cite{Liptser}, the optimal estimator $\hat{x}%
(t)=E_{P^{\theta^{\ast}}}[{x}(t)|\mathcal{Z}_{t}]$ satisfies
(\ref{convex theta optimal solution}). Since $B(t),\ H(t),\ b(t)\ $and$\\ h(t)$
are uniformly bounded, deterministic functions and $\theta^{\ast}$ is bounded,
by Theorem 6.3 (see Chapter 1 in \cite{Yong-Zhou}), the solution $\hat{x}(t)$
to \eqref{convex theta optimal solution} also belongs to $L_{\mathcal{Z}_{t}%
}^{2+\epsilon}(\Omega,\mathbb{P},\mathbb{R})$. It yields that $\hat{x}(t)$\ is
the optimal solution of problem \eqref{convex auxiliary problem} at time
$t\in\lbrack0,T]$. This completes the proof.
\end{proof}

\begin{corollary}
If $\theta^{\ast}(t)$ is adapted to $\mathcal{Z}_{t}$, then $\hat{x}(t)$
satisfies the following equation:
\begin{equation}
\left\{
\begin{array}
[c]{rl}%
d\hat{x}(t) & =(B(t)\hat{x}(t)+b(t)-\theta_{1}^{\ast}(t))dt+P(t)H(t)R(t)^{-1}%
d\hat{I}(t),\\
\hat{x}(0) & =x_{0},
\end{array}
\right.  \label{convex simpler optimal solution}%
\end{equation}
where $P(t)$ satisfies the following Riccati equation:%
\begin{equation}
\left\{
\begin{array}
[c]{rl}
& \frac{dP(t)}{dt}=B(t)P(t)+P(t)B(t)^{\intercal}-P(t)H(t)^{\intercal}%
R(t)^{-1}H(t)P(t)+Q(t),\\
& P(0)=0.
\end{array}
\right.  \label{Riccati}%
\end{equation}

\end{corollary}

Define
\[
A(t,s)=\exp^{\int_{s}^{t}(B(r)-P(r)H(r)^{2}R^{-1}(r))dr}.
\]
$\bar{x}(t)$ is governed by
\begin{equation}
\left\{
\begin{array}
[c]{ll}%
d\bar{x}(t) & =(B(t)\bar{x}(t)+b(t))dt+P(t)H(t)^{\intercal}R(t)^{-1}dI(t),\\
\bar{x}(0) & =x_{0},
\end{array}
\right.  \label{classical Kalman}%
\end{equation}
where
\[
I(t)=m(t)-\int_{0}^{t}(H(s)\bar{x}(s)+h(s))ds.
\]

\begin{corollary}
\label{decomposition} If the optimal $\theta^{\ast}(t)$ adapted to
subfiltration $\mathcal{Z}_{t}$, with equations \eqref{classical Kalman} and
\eqref{convex theta optimal solution}, then the optimal estimator $\hat{x}(t)$ for any time $t\in\lbrack0.T]$ can be expressed as
\begin{equation}
\hat{x}(t)=\bar{x}(t)+\int_{0}^{t}(P(s)H(s)R^{-1}(s)\theta_{2}^{\ast
}(s)-\theta_{1}^{\ast}(s))A(t,s)ds. \label{expression2}%
\end{equation}
where $\bar{x}(t)$ is defined by equation \eqref{classical Kalman}.
\end{corollary}

\begin{remark}
So far, we have only proved the existence of the optimal $\theta^{\ast}$ from the mathematical theory. Since the complexity of the problem is considered in this paper, it is still a problem to be solved how to calculate the optimal $\theta^{\ast}$. In the future, we plan to study the numerical solutions to the robust estimation \eqref{convex robust problem}.

\end{remark}

\section{MMSE under general convex operators {on }$L_{\mathcal{F}}^{p}(\mathbb{P})\label{section-general problem}$}

In section \ref{Kalman-Bucy filtering under uncertainty}, we boil down the
calculation of the Kalman-Bucy filter under uncertainty to solving a minimum
mean square estimation problem under the convex $g$-expectation. The
worst-case prior $P^{\theta^{\ast}}$ is obtained and the corresponding
filtering equation (\ref{convex theta optimal solution}) is deduced.

It is an interesting question whether there are similar results for general
convex operators. So in this section, we investigate the minimum mean square
estimation problem under general convex operators on $L_{\mathcal{F}}%
^{p}(\mathbb{P})$ and obtain the existence and uniqueness results of the MMSE.

\subsection{{General convex operators on }$L_{\mathcal{F}}^{p}(\mathbb{P})$}

For a given probability space $(\Omega,\mathcal{F},\mathbb{P})$, we denote the
set of all $\mathcal{F}$-measurable $p$-th power integrable random variables
by $L^{p}(\Omega,\mathcal{F},\mathbb{P})$. Sometimes we use $L_{\mathcal{F}%
}^{p}(\mathbb{P})$ for short. Let $\mathcal{C}$ be a sub $\sigma$-algebra of
$\mathcal{F}$. $L_{\mathcal{C}}^{p}(\mathbb{P})$ denotes the set of all the
$p$-th power integrable $\mathcal{C}$-measurable random variables. In this
paper, we only consider the case that $1<p\leq2$.

Let $\mathcal{M}$ denote the set of probability measures absolutely continuous
with respect to $\mathbb{P}$. For $P\in\mathcal{M}$, we will use $f^{P}$ to
denote the Radon-Nikodym derivative $\frac{dP}{d\mathbb{P}}$ and $E_{P}%
[\cdot]$ to denote the expectation under $P$. Especially, the expectation
under $\mathbb{P}$ is denoted as $\mathbb{E[\cdot]}$. For a sub $\sigma
$-algebra $\mathcal{C}$ of $\mathcal{F}$ and $P\in\mathcal{M}$, define
$f_{\mathcal{C}}^{P}=\mathbb{E}[f^{P}|\mathcal{C}]$.

\begin{definition}
\label{def2.1} A convex operator is an operator $\rho(\cdot):L_{\mathcal{F}%
}^{p}(\mathbb{P})\mapsto\mathbb{R}$ satisfying

(i) Monotonicity: for any $\xi_{1},$ $\xi_{2}\in L_{\mathcal{F}}%
^{p}(\mathbb{P})$, $\rho(\xi_{1})\geq\rho(\xi_{2})$ if $\xi_{1}\geq\xi_{2}$;

(ii) Constant invariance: $\rho(\xi+c)=\rho(\xi)+c$ for any $\xi\in
L_{\mathcal{F}}^{p}(\mathbb{P})$ and $c\in\mathbb{R}$;

(iii) Convexity: for any $\xi_{1},$ $\xi_{2}\in L_{\mathcal{F}}^{p}%
(\mathbb{P})$ and $\lambda\in\lbrack0,1]$, $\rho(\lambda\xi_{1}+(1-\lambda
)\xi_{2})\leq\lambda\rho(\xi_{1})+(1-\lambda)\rho(\xi_{2})$.
\end{definition}

\begin{definition}
A convex operator $\rho(\cdot)$ is called normalized if $\rho(0)=0$.
\end{definition}

\begin{remark}
In this paper, we will always assume the convex operator is normalized. Moreover, if we define $\rho^{\prime}(\xi)=\rho(-\xi)$, then $\rho^{\prime}(\cdot)$ is a convex risk measure on $L_{\mathcal{F}}^{p}(\mathbb{P})$.
\end{remark}

If $\rho(\cdot)$ is a convex operator, then by Proposition 2.10 and Theorem
2.11 in \cite{KR}, for any random variable $\xi\in L_{\mathcal{F}}%
^{p}(\mathbb{P})$, there exists a set $\mathcal{P}$ such that $\rho(\cdot)$
can be represented as
\[
\rho(\xi)=\sup_{P\in\mathcal{P}}[E_{P}[\xi]-\alpha(P)],
\]
where $\alpha(P):=\sup_{\zeta\in\mathcal{A}_{\rho}}E_{P}[\zeta]$,
$\mathcal{A}_{\rho}:=\{\zeta\in L_{\mathcal{F}}^{p}(\mathbb{P});$ $\rho
(\zeta)\leq0\}$ called acceptance set, $\mathcal{P:}=\{P\in\mathcal{M};$ $f^{P}\in L_{\mathcal{F}%
}^{q}(\mathbb{P}),$ $\alpha(P)<\infty\}$. Moreover, $\mathcal{D}:=\{f^{P};$
$P\in\mathcal{P}\}$ is norm-bounded in $L_{\mathcal{F}}^{q}(\mathbb{P})$ and
$\sigma(L_{\mathcal{F}}^{q}(\mathbb{P}),L_{\mathcal{F}}^{p}(\mathbb{P}%
))$-compact, where $\sigma(L_{\mathcal{F}}^{q}(\mathbb{P}),L_{\mathcal{F}}%
^{p}(\mathbb{P}))$ denotes the weak topology defined on $L_{\mathcal{F}}%
^{q}(\mathbb{P})$ and $\frac{1}{p}+\frac{1}{q}=1$. The set $\mathcal{P}$ is
called the representation set of $\rho(\cdot)$. Since $\alpha(\cdot)$ is a
convex function defined on $\mathcal{M}$,$\ \mathcal{P}$ is a convex set.
%%%%%%%%%%%%%%%%%%%%%%%%%%%%%%%%%%%%%%%%%%%%%%%%%%%%%%%%%%%%%%%%%%%%%%%%%%%%%%%%%%%%%%%%%%%%%%%%%%%%%%%%%%%%%%%%%%%%%%%%%%%%%%%%%%%%%%%%%%%%%%%%%%%%%%%%%%%%%%%%%%
\begin{remark}
Note that $\alpha(P)=\sup_{\zeta\in\mathcal{A}_{\rho}}E_{P}[\zeta]=\sup
_{\zeta\in\mathcal{A}_{\rho}}\mathbb{E}[f^{P}\zeta]$. By abuse of notation, we
sometimes write $\alpha(f^{P})$ instead of $\alpha(P)$.
\end{remark}

\begin{definition}
\label{def3.1} The set $\mathcal{P}$ is called stable if for any element
$P\in\mathcal{P}$ and any sub $\sigma$-algebra $\mathcal{C}$ of $\mathcal{F}$,
$\frac{f^{P}}{f_{\mathcal{C}}^{P}}$ still lies in the set $\mathcal{D}$.
\end{definition}

\begin{definition}
\label{def3.2} A convex operator $\rho(\cdot)$ is called stable, if its
representation set $\mathcal{P}$ is stable.
\end{definition}

\begin{definition}
A convex operator $\rho(\cdot)$ is called proper if all the elements in its
representation set $\mathcal{P}$ are equivalent to $\mathbb{P}$.
\end{definition}

For a given $\xi\in L_{\mathcal{F}}^{4p}(\mathbb{P})$, when we only know the
information $\mathcal{C}$, we want to find the minimum mean square estimator
of $\xi$ under the convex operator $\rho(\cdot)$. In more detail, we will
solve the following optimization problem:

%\begin{problem}
\textbf{Problem:}
For a given $\xi\in L_{\mathcal{F}}^{4p}(\mathbb{P})$, find a $\hat{\eta}\in
L_{\mathcal{C}}^{2p}(\mathbb{P})$ such that
\begin{equation}
\rho(\xi-\hat{\eta})^{2}=\inf_{\eta\in L_{\mathcal{C}}^{2p}(\mathbb{P})}%
\rho(\xi-\eta)^{2}. \label{problem}%
\end{equation}

%\end{problem}

The optimal solution $\hat{\eta}$ of (\ref{problem}) is called the minimum
mean square estimator and we will denote it by $\rho(\xi|\mathcal{C})$.
%%%%%%%%%%%%%%%%%%%%%%%%%%%%%%%%%%%%%%%%%%%%%%%%%%%%%%%%%%%%%%%%%%%%%%%%%%%%%%%%%%%%%%%%%%%%%%%%%%%%%%%%%%%%%%%%%%%%%%%%%%%%%%%%%%%%%%%%%%%%%%%%%%%%%%%%%%%%%%%%
\begin{remark}
If we set $\mathcal{C=Z}_{t}$ and $p=1+\frac{\epsilon}{2}$ with $\epsilon
\in(0,1)$, then $L_{\mathcal{C}}^{2p}(\mathbb{P})$ is just the space
$L_{\mathcal{Z}_{t}}^{2+\epsilon}(\Omega,\mathbb{P},\mathbb{R}^{n})$ in
subsection \ref{problem formulation--kalman-bucy}.
\end{remark}

\subsection{{Existence and uniqueness results}}

In this section, we study the existence and uniqueness of the minimum mean
square estimator for problem \ref{problem}. We first give the following assumption.

\begin{assumption}
\label{Asum-proper--stable} The convex operator $\rho(\cdot)$ is stable and proper.
\end{assumption}
\subsubsection{Existence}

\begin{lemma}
\label{lem-pre} For any given real number $\gamma\geq2$, if $\xi\in
L_{\mathcal{F}}^{\gamma p}(\mathbb{P})$, then we have $\sup\limits_{P\in\mathcal{P}%
}E_{P}[\xi^{\frac{\gamma p}{2}}]<\infty$.
\end{lemma}

\begin{proof}
Since $\{f^{P};P\in\mathcal{P}\}$ is normed bounded in $L_{\mathcal{F}}%
^{q}(\mathbb{P})$ and $1<p\leq2$, we have
\[
\sup_{P\in\mathcal{P}}E_{P}[\xi^{\frac{\gamma p}{2}}]=\sup_{P\in\mathcal{P}%
}\mathbb{E}[f^{P}\xi^{\frac{\gamma p}{2}}]\leq\sup_{P\in\mathcal{P}}%
||f^{P}||_{L^{q}}||\xi^{\frac{\gamma p}{2}}||_{L^{p}}\leq\sup_{P\in
\mathcal{P}}||f^{P}||_{L^{q}}(||\xi||_{L^{\gamma p}})^{\gamma}<\infty.
\]
This completes the proof.
\end{proof}

\begin{lemma}
\label{pro-stable} Suppose that Assumption \ref{Asum-proper--stable} holds.
Then for any $P\in\mathcal{P}$, $\xi\in L_{\mathcal{F}}^{p}(\mathbb{P})$ and
sub $\sigma$-algebra $\mathcal{C}$ of $\mathcal{F}$, there exists a $\bar
{P}\in\mathcal{P}$ such that $E_{\bar{P}}[\xi]=\mathbb{E}[E_{P}[\xi
|\mathcal{C}]]$.
\end{lemma}

\begin{proof}
It is obvious that
\[
\mathbb{E}[E_{P}[\xi|\mathcal{C}]]=\mathbb{E}[\frac{\mathbb{E}[\xi
f^{P}|\mathcal{C}]}{\mathbb{E}[f^{P}|\mathcal{C}]}]=\mathbb{E}[\mathbb{E}%
[\xi\frac{f^{P}}{f_{\mathcal{C}}^{P}}|{\mathcal{C}}]]=\mathbb{E}[\xi
\frac{f^{P}}{f_{\mathcal{C}}^{P}}].
\]
By Definition \ref{def3.1}, there exists a $\bar{P}\in\mathcal{P}$ such that
$\frac{d\bar{P}}{d\mathbb{P}}=\frac{f^{P}}{f_{\mathcal{C}}^{P}}$ which implies
that $E_{\bar{P}}[\xi]=\mathbb{E}[E_{P}[\xi|\mathcal{C}]]$. This completes the proof.
\end{proof}

\begin{proposition}
\label{pro-restrict} Suppose that Assumption \ref{Asum-proper--stable} holds.
If $\xi\in L_{\mathcal{F}}^{4p}(\mathbb{P})$, then there exists a constant $M$
such that for any probability measure $P\in\mathcal{P,}$
\[
\inf_{\eta\in L_{\mathcal{C}}^{2p}(\mathbb{P})}[E_{P}[(\xi-\eta)^{2}%
]-\alpha(P)]=\inf_{\eta\in L_{\mathcal{C}}^{2p,M}(\mathbb{P})}[E_{P}[(\xi
-\eta)^{2}]-\alpha(P)],
\]
where $L_{\mathcal{C}}^{2p,M}(\mathbb{P})$ denotes all the elements in
$L_{\mathcal{C}}^{2p}(\mathbb{P})$ which are norm-bounded by the constant $M$.
\end{proposition}

\begin{proof}
Set $\mathbb{G}=\{E_{P}[\xi|\mathcal{C}];P\in\mathcal{P}\}$. For any
$P\in\mathcal{P}$, we have $\mathbb{E}[(E_{P}[\xi|\mathcal{C}])^{2p}%
]\leq\mathbb{E}[E_{P}[\xi^{2p}|\mathcal{C}]]$. By Lemma \ref{pro-stable},
there exists a $\bar{P}\in\mathcal{P}$ such that $E_{\bar{P}}[\xi
^{2p}]=\mathbb{E}[E_{P}[\xi^{2p}|\mathcal{C}]]$. By Lemma \ref{lem-pre}, there
exists a constant $M_{1}$ such that $\sup_{P\in\mathcal{P}}E_{P}[\xi^{2p}]\leq
M_{1}$. Then $\mathbb{G}\subset L_{\mathcal{C}}^{2p,M}(\mathbb{P})$ where
$M=M_{1}^{\frac{1}{2p}}$. Since $1<p\leq2$, it is obvious that
\[
\mathbb{G}\subset L_{\mathcal{C}}^{2p}(\mathbb{P})\subset\left(
\underset{0<\epsilon\leq2}{\bigcup}L_{\mathcal{C}}^{2+\epsilon}(\mathbb{P}%
)\right)  .
\]
By the project property of conditional expectations, for any $P\in\mathcal{P}$
and $\eta\in L_{\mathcal{C}}^{2+\epsilon}(\mathbb{P})$ with $\epsilon\in
(0,2]$, we have that
\[
E_{P}[(\xi-E_{P}[\xi|\mathcal{C}])^{2}]\leq E_{P}[(\xi-\eta)^{2}]
\]
which leads to%
\[
\inf_{\eta\in L_{\mathcal{C}}^{2p}(\mathbb{P})}[E_{P}[(\xi-\eta)^{2}%
]-\alpha(P)]\geq\inf_{\eta^{\prime}\in\mathbb{G}}[E_{P}[(\xi-\eta^{\prime
})^{2}]-\alpha(P)].
\]
On the other hand, the inverse inequality is obviously true. Then the
following equality holds for any $P\in\mathcal{P}$:
\[
\inf_{\eta\in L_{\mathcal{C}}^{2p}(\mathbb{P})}[E_{P}[(\xi-\eta)^{2}%
]-\alpha(P)]=\inf_{\eta\in\mathbb{G}}[E_{P}[(\xi-\eta)^{2}]-\alpha(P)].
\]
Since $\mathbb{G}\subset L_{\mathcal{C}}^{2p,M}(\mathbb{P})\subset
L_{\mathcal{C}}^{2p}(\mathbb{P})$, it follows that
\[
\inf_{\eta\in L_{\mathcal{C}}^{2p}(\mathbb{P})}[E_{P}[(\xi-\eta)^{2}%
]-\alpha(P)]=\inf_{\eta\in L_{\mathcal{C}}^{2p,M}(\mathbb{P})}[E_{P}[(\xi
-\eta)^{2}]-\alpha(P)].
\]
This completes the proof.
\end{proof}

By Proposition \ref{pro-restrict}, it is easy to see that%
\[
\sup_{P\in\mathcal{P}}\inf_{\eta\in L_{\mathcal{C}}^{2p}(\mathbb{P})}%
[E_{P}[(\xi-\eta)^{2}]-\alpha(P)]=\sup_{P\in\mathcal{P}}\inf_{\eta\in
L_{\mathcal{C}}^{2p,M}(\mathbb{P})}[E_{P}[(\xi-\eta)^{2}]-\alpha(P)].
\]

\begin{lemma}
\label{lsc} $\alpha(\cdot)$ is a lower semi-continuous (l.s.c.) function on
the topology space $(\mathcal{D},\sigma(L_{\mathcal{F}}^{q}(\mathbb{P}%
),L_{\mathcal{F}}^{p}(\mathbb{P})))$.
\end{lemma}

\begin{proof}
For any fixed random variable $\zeta\in\mathcal{A}_{\rho}$, define
$\varphi(\zeta,f^{P})=\mathbb{E}[f^{P}\zeta]$ where $f^{P}$ belongs to
$\mathcal{D}$. Then $\varphi(\zeta,\cdot)$ is a continuous function on the
topology space $(\mathcal{D},\sigma(L_{\mathcal{F}}^{q}(\mathbb{P}%
)\\ , L_{\mathcal{F}}^{p}(\mathbb{P})))$. Since $\alpha(f^{P})=\sup_{\zeta
\in\mathcal{A}_{\rho}}\varphi(\zeta,f^{P})$, based on lower-semicontinuous
definition B.1.1 in Pham \cite{Pham}, then $\alpha(P)$ is a l.s.c. function on
the topology space $(\mathcal{D},\sigma(L_{\mathcal{F}}^{q}(\mathbb{P}%
),L_{\mathcal{F}}^{p}(\mathbb{P})))$. This completes the proof.
\end{proof}

For $\xi\in L_{\mathcal{F}}^{4p}(\mathbb{P})$, $\eta\in L_{\mathcal{C}}%
^{2p}(\mathbb{P})$ and $P\in\mathcal{P}$, define%
\[
l(\xi,\eta,f^{P})=\mathbb{E}[f^{P}(\xi-\eta)^{2}]-\alpha(f^{P})\text{.}%
\]

\begin{lemma}
\label{lem-lsc-usc} For any random variables $\xi\in L_{\mathcal{F}}%
^{4p}(\mathbb{P})$ and $\eta\in L_{\mathcal{C}}^{2p}(\mathbb{P})$, $l(\xi
,\eta,\cdot)$ is an upper semi-continuous (u.s.c.) function on the topology
space $(\mathcal{D},\sigma(L_{\mathcal{F}}^{q}(\mathbb{P}),L_{\mathcal{F}}%
^{p}(\mathbb{P})))$.
\end{lemma}

\begin{proof}
Since $\xi\in L_{\mathcal{F}}^{4p}(\mathbb{P})$ and $\eta\in L_{\mathcal{C}%
}^{2p}(\mathbb{P})$, then $(\xi-\eta)^{2}\in L_{\mathcal{F}}^{p}(\mathbb{P})$
which implies that $\mathbb{E}[f^{P}(\xi-\eta)^{2}]$ is a continuous function
with respect to $f^{P}$ on the topology space $(\mathcal{D},\sigma
(L_{\mathcal{F}}^{q}(\mathbb{P}),L_{\mathcal{F}}^{p}(\mathbb{P})))$. By Lemma
\ref{lsc}, $\alpha(\cdot)$ is a l.s.c. function on the topology space
$(\mathcal{D},\sigma(L_{\mathcal{F}}^{q}(\mathbb{P}),L_{\mathcal{F}}%
^{p}(\mathbb{P})))$. Thus, $l(\xi,\eta,\cdot)$ is an u.s.c. function on the
topology space $(\mathcal{D},\sigma(L_{\mathcal{F}}^{q}(\mathbb{P}%
),L_{\mathcal{F}}^{p}(\mathbb{P})))$. This completes the proof.
\end{proof}

\begin{proposition}
\label{P exist} Suppose that Assumption \ref{Asum-proper--stable} holds. Then
for a given $\xi\in L_{\mathcal{F}}^{4p}(\Omega,\mathbb{P})$, there exists a
$\hat{P}\in\mathcal{P}$ such that%
\[
\inf_{\eta\in L_{\mathcal{C}}^{2p,M}(\mathbb{P})}[E_{\hat{P}}[(\xi-\eta
)^{2}]-\alpha(\hat{P})]=\sup_{P\in\mathcal{P}}\inf_{\eta\in L_{\mathcal{C}%
}^{2p,M}(\mathbb{P})}[E_{P}[(\xi-\eta)^{2}]-\alpha(P)],
\]
where $M$ is the constant given in Proposition \ref{pro-restrict}.
\end{proposition}

\begin{proof}
Define
\[
\beta=\sup_{P\in\mathcal{P}}\inf_{\eta\in L_{\mathcal{C}}^{2p,M}(\mathbb{P}%
)}[E_{P}[(\xi-\eta)^{2}]-\alpha(P)]=\sup_{f^{P}\in\mathcal{D}}\inf_{\eta\in
L_{\mathcal{C}}^{2p,M}(\mathbb{P})}[\mathbb{E}[f^{P}(\xi-\eta)^{2}%
]-\alpha(f^{P})].
\]
Take a sequence $\{f^{P_{n}};P_{n}\in\mathcal{P}\}_{n\geq1}$ such that
\[
\inf_{\eta\in L_{\mathcal{C}}^{2p,M}(\mathbb{P})}[\mathbb{E}[f^{P_{n}}%
(\xi-\eta)^{2}]-\alpha(f^{P_{n}})]\geq\beta-\frac{1}{2^{n}}.
\]
Since $\mathcal{D}$ is a weakly compact set, we can take a subsequence
$\{f^{P_{n_{i}}}\}_{i\geq1}$ which weakly converges to some $f^{\hat{P}}\in
L_{\mathcal{F}}^{q}(\mathbb{P})$. Therefore, $\hat{P}\in\mathcal{P}$ and there
exists a sequence $\{f^{\tilde{P}_{i}}\in conv(f^{P_{n_{i}}},f^{P_{n_{i+1}}%
},...)\}_{i\geq1}$ such that $f^{\tilde{P}_{i}}$ converges to $f^{\hat{P}}$ in
$L_{\mathcal{F}}^{q}(\mathbb{P})$-norm by Theorem \ref{Mazur} in the Appendix.

For any $\eta\in L_{\mathcal{C}}^{2p,M}(\mathbb{P})$ and $i\in\mathbb{N}$,%
\[
\lim_{i\rightarrow\infty}\mathbb{E}|f^{\tilde{P}_{i}}(\xi-\eta)^{2}-f^{\hat
{P}}(\xi-\eta)^{2}|\leq\lim_{i\rightarrow\infty}||(f^{\tilde{P}_{i}}%
-f^{\hat{P}})||_{L^{q}(\mathbb{P})}||(\xi-\eta)^{2}||_{L^{p}(\mathbb{P})}=0,
\]
which leads to%
\[
\lim_{i\rightarrow\infty}\mathbb{E}[f^{\tilde{P}_{i}}(\xi-\eta)^{2}%
]=\mathbb{E}[f^{\hat{P}}(\xi-\eta)^{2}].
\]
On the other hand,
\[%
\begin{array}
[c]{rl}%
|\alpha(f^{\hat{P}})-\alpha(f^{\tilde{P}_{i}})| & =|\sup\limits_{\zeta
\in\mathcal{A}_{\rho}}\mathbb{E}[f^{\hat{P}}\zeta]-\sup\limits_{\zeta
\in\mathcal{A}_{\rho}}\mathbb{E}[f^{\tilde{P}_{i}}\zeta]|\leq\sup
\limits_{\zeta\in\mathcal{A}_{\rho}}\mathbb{E}[|(f^{\hat{P}}-f^{\tilde{P}_{i}%
})\zeta|]\\
& \leq\sup\limits_{\zeta\in\mathcal{A}_{\rho}}||(f^{\tilde{P}_{i}}-f^{\hat{P}%
})||_{L^{q}(\mathbb{P})}||\zeta||_{L^{p}(\mathbb{P})}.
\end{array}
\]
Then,
\[
\lim_{i\rightarrow\infty}[\mathbb{E}[[f^{\tilde{P}_{i}}(\xi-\eta)^{2}%
]-\alpha(f^{\tilde{P}_{i}})]=\mathbb{E}[f^{\hat{P}}(\xi-\eta)^{2}%
]-\alpha(f^{\hat{P}})].
\]
Since
\[
\lbrack\mathbb{E}[f^{\tilde{P}_{i}}(\xi-\eta)^{2}]-\alpha(f^{\tilde{P}_{i}%
})]\geq\inf_{\tilde{\eta}\in L_{\mathcal{C}}^{2p,M}(\mathbb{P})}%
[\mathbb{E}[f^{\tilde{P}_{i}}(\xi-\tilde{\eta})^{2}]-\alpha(f^{\tilde{P}_{i}%
})]
\]
for any $\eta\in L_{\mathcal{C}}^{2p,M}(\mathbb{P})$, we have that
\[
\lim_{i\rightarrow\infty}[\mathbb{E}[f^{\tilde{P}_{i}}(\xi-\eta)^{2}%
]-\alpha(f^{\tilde{P}_{i}})]\geq\limsup_{i\rightarrow\infty}\inf_{\tilde{\eta
}\in L_{\mathcal{C}}^{2p,M}(\mathbb{P})}[\mathbb{E}[f^{\tilde{P}_{i}}%
(\xi-\tilde{\eta})^{2}]-\alpha(f^{\tilde{P}_{i}})].
\]
It yields that%
\begin{equation}%
\begin{array}
[c]{rl}
& \inf_{\eta\in L_{\mathcal{C}}^{2p,M}(\mathbb{P})}[\mathbb{E}[f^{\hat{P}}%
(\xi-\eta)^{2}]-\alpha(\hat{P})]\\
= & \inf_{\eta\in L_{\mathcal{C}}^{2p,M}(\mathbb{P})}\lim_{i\rightarrow\infty
}[\mathbb{E}[f^{\tilde{P}_{i}}(\xi-\eta)^{2}]-\alpha(f^{\tilde{P}_{i}})]\\
\geq & \limsup_{i\rightarrow\infty}\inf_{\tilde{\eta}\in L_{\mathcal{C}%
}^{2p,M}(\mathbb{P})}[\mathbb{E}[f^{\tilde{P}_{i}}(\xi-\tilde{\eta}%
)^{2}]-\alpha(f^{\tilde{P}_{i}})].
\end{array}
\label{eq-exist-1}%
\end{equation}

As $\alpha(\cdot)$ is a convex function and $f^{\tilde{P}_{i}}\in
conv(f^{P_{n_{i}}},f^{P_{n_{i+1}}},...)$, we have
\begin{equation}%
\begin{array}
[c]{rl}%
\limsup\limits_{i\rightarrow\infty}\inf\limits_{\tilde{\eta}\in L_{\mathcal{C}%
}^{2p,M}(\mathbb{P})}[\mathbb{E}[f^{\tilde{P}_{i}}(\xi-\tilde{\eta}%
)^{2}]-\alpha(f^{\tilde{P}_{i}})]\geq\beta. &
\end{array}
\label{eq-exist-2}%
\end{equation}
Combining (\ref{eq-exist-1}) and (\ref{eq-exist-2}), we obtain the result.
\end{proof}

\begin{corollary}
\label{P exist2} Suppose that Assumption \ref{Asum-proper--stable} holds. Then
for a given $\xi\in L_{\mathcal{F}}^{4p}(\mathbb{P})$, there exists a $\hat
{P}\in\mathcal{P}$ such that
\[
\inf_{\eta\in L_{\mathcal{C}}^{2p}(\mathbb{P})}[E_{\hat{P}}[(\xi-\eta
)^{2}]-\alpha(\hat{P})]=\sup_{P\in\mathcal{P}}\inf_{\eta\in L_{\mathcal{C}%
}^{2p}(\mathbb{P})}[E_{P}[(\xi-\eta)^{2}]-\alpha(P)].
\]

\end{corollary}

\begin{proof}
Choose $\hat{P}$ as in Proposition \ref{P exist}. By Propositions
\ref{pro-restrict} and \ref{P exist}, the following relations hold
\[%
\begin{array}
[c]{rl}
& \sup\limits_{P\in\mathcal{P}}\inf\limits_{\eta\in L_{\mathcal{C}}%
^{2p}(\mathbb{P})}[E_{P}[(\xi-\eta)^{2}]-\alpha(P)]=\sup\limits_{P\in
\mathcal{P}}\inf\limits_{\eta\in L_{\mathcal{C}}^{2p,M}(\mathbb{P})}%
[E_{P}[(\xi-\eta)^{2}]-\alpha(P)]\\
& =\inf\limits_{\eta\in L_{\mathcal{C}}^{2p,M}(\mathbb{P})}[E_{\hat{P}}%
[(\xi-\eta)^{2}]-\alpha(\hat{P})]=\inf\limits_{\eta\in L_{\mathcal{C}}%
^{2p}(\mathbb{P})}[E_{\hat{P}}[(\xi-\eta)^{2}]-\alpha(\hat{P})].
\end{array}
\]
This completes the proof.
\end{proof}

\begin{theorem}
[Existence theorem]\label{existence} Suppose that Assumption
\ref{Asum-proper--stable} holds. Then there exists a $\hat{\eta}\in
L_{\mathcal{C}}^{2p}(\mathbb{P})$ which solves problem \eqref{problem}.
\end{theorem}

\begin{proof}
For given $\xi\in L_{\mathcal{F}}^{4p}(\mathbb{P})$, $\eta\in L_{\mathcal{C}%
}^{2p}(\mathbb{P})$ and $P\in\mathcal{P}$, it is easy to check that
$l(\xi,\cdot,f^{P})$ is convex on $L_{\mathcal{C}}^{2p}(\mathbb{P})$ and
$l(\xi,\eta,\cdot)$ is concave on $L_{\mathcal{F}}^{q}(\mathbb{P})$. As
$\mathcal{D}$ is $\sigma(L_{\mathcal{F}}^{q}(\mathbb{P}),L_{\mathcal{F}}%
^{p}(\mathbb{P}))$-compact and $l(\xi,\eta,\cdot)$ is u.s.c on the topology
space $(L_{\mathcal{F}}^{q}(\mathbb{P}),\sigma(L_{\mathcal{F}}^{q}%
(\mathbb{P})\\ , L_{\mathcal{F}}^{p}(\mathbb{P})))$ by Lemma \ref{lem-lsc-usc}, we
have
\begin{align*}
\inf_{\eta\in L_{\mathcal{C}}^{2p}(\mathbb{P})}\max_{P\in\mathcal{P}}%
[E_{P}[(\xi-\eta)^{2}]-\alpha(P)]  &  =\max_{P\in\mathcal{P}}\inf_{\eta\in
L_{\mathcal{C}}^{2p}(\mathbb{P})}[E_{P}[(\xi-\eta)^{2}]-\alpha(P)];\\
\inf_{\eta\in L_{\mathcal{C}}^{2p,M}(\mathbb{P})}\max_{P\in\mathcal{P}}%
[E_{P}[(\xi-\eta)^{2}]-\alpha(P)]  &  =\max_{P\in\mathcal{P}}\inf_{\eta\in
L_{\mathcal{C}}^{2p,M}(\mathbb{P})}[E_{P}[(\xi-\eta)^{2}]-\alpha(P)]
\end{align*}
by Proposition \ref{P exist}, Corollary \ref{P exist2} and Theorem
\ref{minmax} in the Appendix. With the help of Proportion \ref{pro-restrict},
\[
\inf_{\eta\in L_{\mathcal{C}}^{2p}(\mathbb{P})}\max_{P\in\mathcal{P}}%
[E_{P}[(\xi-\eta)^{2}]-\alpha(P)]=\inf_{\eta\in L_{\mathcal{C}}^{2p,M}%
(\mathbb{P})}\max_{P\in\mathcal{P}}[E_{P}[(\xi-\eta)^{2}]-\alpha(P)].
\]

Therefore, we can take a sequence $\{\eta_{n};n\in\mathbb{N}\}\subset
L_{\mathcal{C}}^{2p,M}(\mathbb{P})$ such that
\[
\rho(\xi-\eta_{n})^{2}<\beta+\frac{1}{2^{n}},
\]
where $\beta:=\inf_{\eta\in L_{\mathcal{C}}^{2p}(\mathbb{P})}\rho(\xi
-\eta)^{2}$. Since $L_{\mathcal{C}}^{2p,M}(\mathbb{P})$ is a weakly compact
set, we can take a subsequence $\{\eta_{n_{i}}\}_{i\in\mathbb{N}}$ of
$\{\eta_{n}\}_{n\in\mathbb{N}}$ which weakly converges to some $\hat{\eta}\in
L_{\mathcal{C}}^{2p,M}(\mathbb{P})$. By theorem \ref{Mazur} in the Appendix,
there exists a sequence $\{\tilde{\eta}_{i}\in\text{conv}(\eta_{n_{i}}%
,\eta_{n_{i+1}},\cdots)\}_{i\in\mathbb{N}}$ such that $\tilde{\eta}_{i}$
converges to $\hat{\eta}$ in $L_{\mathcal{C}}^{2p}(\mathbb{P})$-norm. Then
\begin{equation}%
\begin{array}
[c]{r@{}l}%
\rho(\xi-\hat{\eta})^{2} & =\rho(\xi-\tilde{\eta}_{i}+\tilde{\eta}_{i}%
-\hat{\eta})^{2}\\
& =\sup_{P\in\mathcal{P}}[E_{P}[(\xi-\tilde{\eta}_{i})^{2}+(\tilde{\eta}%
_{i}-\hat{\eta})^{2}+2(\xi-\tilde{\eta}_{i})(\tilde{\eta}_{i}-\hat{\eta
})]-\alpha(P)]\\
& \leq\sup_{P\in\mathcal{P}}[E_{P}[(\xi-\tilde{\eta}_{i})^{2}]-\alpha
(P)]+\sup_{P\in\mathcal{P}}E_{P}[(\tilde{\eta}_{i}-\hat{\eta})^{2}\\
&+2(\xi-\tilde{\eta}_{i})(\tilde{\eta}_{i}-\hat{\eta})]\\
& =\rho(\xi-\tilde{\eta}_{i})^{2}+\sup_{P\in\mathcal{P}}E_{P}[-(\tilde{\eta
}_{i}-\hat{\eta})^{2}+2(\xi-\hat{\eta})(\tilde{\eta}_{i}-\hat{\eta})]\\
& \leq\beta+\frac{1}{2^{i-1}}+2\sup_{P\in\mathcal{P}}||f^{P}||_{L^{q}}%
||\tilde{\eta}_{i}-\hat{\eta}||_{L^{2p}}(1+||\xi-\hat{\eta}||_{L^{2p}}).
\end{array}
\label{3.1}%
\end{equation}
As (\ref{3.1}) holds for any $i\geq1$, we have that $\rho(\xi-\hat{\eta}%
)^{2}=\beta$.
\end{proof}

\subsubsection{Uniqueness}

In this subsection, we prove that the optimal solution of problem
\eqref{problem} is unique.

\begin{proposition}
\label{Prop--unique}Suppose that Assumption \ref{Asum-proper--stable} holds.
If $\hat{\eta}$ is an optimal solution of problem \eqref{problem}, then there
exists a $\hat{P}\in\mathcal{P}$ such that $\hat{\eta}=E_{\hat{P}}%
[\xi|\mathcal{C}]$.
\end{proposition}

\begin{proof}
If $\hat{\eta}$ is an optimal solution of problem \eqref{problem}, then there
exists a $\hat{P}\in\mathcal{P}$ such that
\begin{align*}
&  \sup_{P\in\mathcal{P}}[\mathbb{E}[f_{P}(\xi-\hat{\eta})^{2}]-\alpha(P)]\\
&  =\max_{P\in\mathcal{P}}[\mathbb{E}[f_{P}(\xi-\hat{\eta})^{2}]-\alpha(P)]\\
&  =\inf_{\eta\in L_{\mathcal{C}}^{2p}(\mathbb{P})}\max_{P\in\mathcal{P}%
}[\mathbb{E}[f_{P}(\xi-\eta)^{2}]-\alpha(P)]\\
&  =\max_{P\in\mathcal{P}}\inf_{\eta\in L_{\mathcal{C}}^{2p}(\mathbb{P}%
)}[\mathbb{E}[f_{P}(\xi-\eta)^{2}]-\alpha(P)]\\
&  =\inf_{\eta\in L_{\mathcal{C}}^{2p}(\mathbb{P})}[\mathbb{E}[f^{\hat{P}}%
(\xi-\eta)^{2}]-\alpha(\hat{P})]
\end{align*}
by Corollary \ref{P exist2}, Theorem \ref{existence} and Theorem \ref{minmax}
in the Appendix. Thus, by Theorem \ref{saddle point} in the Appendix, $(\hat{\eta
},\hat{P})$ is a saddle point, i.e., for $\forall P\in\mathcal{P},\eta\in L_{\mathcal{C}%
}^{2p}(\mathbb{P})$, we have
\[
\mathbb{E}[f^{P}(\xi-\hat{\eta})^{2}]-\alpha(P)\leq\mathbb{E}[f^{\hat{P}}%
(\xi-\hat{\eta})^{2}]-\alpha(\hat{P})\leq\mathbb{E}[f^{\hat{P}}(\xi-\eta
)^{2}]-\alpha(\hat{P}).
\]

This shows that if $\hat{\eta}$ is an optimal solution, then there exists a
$\hat{P}\in\mathcal{P}$ such that $\hat{\eta}=E_{\hat{P}}[\xi|\mathcal{C}]$ by
the project property of conditional expectations.
\end{proof}

\begin{theorem}
[Uniqueness theorem]\label{unique} Suppose that Assumption
\ref{Asum-proper--stable} holds.\\ Then, the optimal solution of problem
\eqref{problem} is unique.
\end{theorem}

\begin{proof}
Suppose that there exist two optimal solutions $\hat{\eta}_{1}$ and $\hat
{\eta}_{2}$. Denote the corresponding probabilities in Proposition
\ref{Prop--unique} by $\hat{P}_{1}$ and $\hat{P}_{2}$ respectively. Then
$\hat{\eta}_{1}=E_{\hat{P}_{1}}[\xi|\mathcal{C}]$ and $\hat{\eta}_{2}%
=E_{\hat{P}_{2}}[\xi|\mathcal{C}]$. For $\lambda\in(0,1)$, set
\begin{align*}
P^{\lambda}  &  =\lambda\hat{P}_{1}+(1-\lambda)\hat{P}_{2},\\
\lambda_{\hat{P}_{1}}  &  =\lambda E_{P^{\lambda}}[\frac{d\hat{P}_{1}%
}{dP^{\lambda}}|\mathcal{C}]\text{,}\\
\lambda_{\hat{P}_{2}}  &  =(1-\lambda)E_{P^{\lambda}}[\frac{d\hat{P}_{2}%
}{dP^{\lambda}}|\mathcal{C}]\text{.}%
\end{align*}
It is easy to verify that $\lambda_{\hat{P}_{1}}+\lambda_{\hat{P}_{2}}=1$ and
$E_{P^{\lambda}}[\xi|\mathcal{C}]=\lambda_{\hat{P}_{1}}\hat{\eta}_{1}%
+\lambda_{\hat{P}_{2}}\hat{\eta}_{2}$. Noticing that $E_{\hat{P}_{i}}[\xi
-\hat{\eta}_{i}|\mathcal{C}]=0$, $i=1,2$, then we have the following inequality
(Details of the calculation can be found in Lemma \ref{calculate} in the Appendix):%
\[%
\begin{array}
[c]{l@{}l}
& E_{P^{\lambda}}[(\xi-E_{P^{\lambda}}[\xi|\mathcal{C}])^{2}]-\alpha
(P^{\lambda})\\
= & E_{P^{\lambda}}[(\xi-\lambda_{\hat{P}_{1}}\hat{\eta}_{1}-\lambda_{\hat
{P}_{2}}\hat{\eta}_{2})^{2}]-\alpha(P^{\lambda})\\
= & E_{P^{\lambda}}[(\lambda_{\hat{P}_{1}}(\xi-\hat{\eta}_{1})+\lambda
_{\hat{P}_{2}}(\xi-\hat{\eta}_{2}))^{2}]-\alpha(P^{\lambda})\\
= & E_{P^{\lambda}}[\lambda_{\hat{P}_{1}}(\xi-\hat{\eta}_{1})^{2}%
+\lambda_{\hat{P}_{2}}(\xi-\hat{\eta}_{2})^{2}-\lambda_{\hat{P}_{1}}%
\lambda_{\hat{P}_{2}}(\hat{\eta}_{1}-\hat{\eta}_{2})^{2}]-\alpha(P^{\lambda
})\\
= & \lambda E_{\hat{P}_{1}}[(\xi-\hat{\eta}_{1})^{2}-\lambda_{\hat{P}_{2}%
}((\xi-\hat{\eta}_{1})^{2}-(\xi-\hat{\eta}_{2})^{2}+(\hat{\eta}_{1}-\hat{\eta
}_{2})^{2})+\lambda_{\hat{P}_{2}}^{2}(\hat{\eta}_{1}-\hat{\eta}_{2})^{2}]\\
& +(1-\lambda)E_{\hat{P}_{2}}[\lambda_{\hat{P}_{1}}((\xi-\hat{\eta}_{1}%
)^{2}-(\xi-\hat{\eta}_{2})^{2}-(\hat{\eta}_{1}-\hat{\eta}_{2})^{2})+(\xi
-\hat{\eta}_{2})^{2}+\lambda_{\hat{P}_{1}}^{2}(\hat{\eta}_{1}-\hat{\eta}%
_{2})^{2}]\\
&-\alpha(P^{\lambda})\\
= & \lambda E_{\hat{P}_{1}}[(\xi-\hat{\eta}_{1})^{2}]+\lambda E_{\hat{P}_{1}%
}[\lambda_{\hat{P}_{2}}^{2}(\hat{\eta}_{1}-\hat{\eta}_{2})^{2}])+(1-\lambda
)E_{\hat{P}_{2}}[(\xi-\hat{\eta}_{2})^{2}]-\alpha(P^{\lambda})\\
&+(1-\lambda)E_{\hat{P}_{2}}[\lambda_{\hat{P}_{1}}^{2}(\hat{\eta}_{1}-\hat{\eta}_{2})^{2}].
\end{array}
\]
Set $\beta=\inf_{\eta\in L_{\mathcal{C}}^{2p}(\mathbb{P})}\rho(\xi-\eta)^{2}$.
By the above equation and the convexity of $\alpha(\cdot)$,
\begin{equation}%
\begin{array}
[c]{l@{}l}
& E_{P^{\lambda}}[(\xi-E_{P^{\lambda}}[\xi|\mathcal{C}])^{2}]-\alpha
(P^{\lambda})\\
\geq & \lambda E_{\hat{P}_{1}}[(\xi-\hat{\eta}_{1})^{2}]+(1-\lambda)E_{\hat
{P}_{2}}[(\xi-\hat{\eta}_{2})^{2}]-[\lambda\alpha(\hat{P}_{1})+(1-\lambda
)\alpha(\hat{P}_{2})]\\
& +\lambda E_{\hat{P}_{1}}[\lambda_{\hat{P}_{2}}^{2}(\hat{\eta}_{1}-\hat{\eta
}_{2})^{2}]+(1-\lambda)E_{\hat{P}_{2}}[\lambda_{\hat{P}_{1}}^{2}(\hat{\eta
}_{1}-\hat{\eta}_{2})^{2}]\\
= & \beta+\lambda E_{\hat{P}_{1}}[\lambda_{\hat{P}_{2}}^{2}(\hat{\eta}%
_{1}-\hat{\eta}_{2})^{2}]+(1-\lambda)E_{\hat{P}_{2}}[\lambda_{\hat{P}_{1}}%
^{2}(\hat{\eta}_{1}-\hat{\eta}_{2})^{2}]\\
\geq & \beta.
\end{array}
\label{eq-unique-1}%
\end{equation}

On the other hand, since $(\hat{\eta}_{1},\hat{P}_{1})$ is a saddle point, we
have
\[
E_{P^{\lambda}}[(\xi-E_{P^{\lambda}}[\xi|\mathcal{C}])^{2}]-\alpha(P^{\lambda
})\leq E_{P^{\lambda}}[(\xi-\hat{\eta}_{1})^{2}]-\alpha(P^{\lambda})\leq
E_{\hat{P}_{1}}[(\xi-\hat{\eta}_{1})^{2}]-\alpha(\hat{P}_{1})=\beta.
\]
It yields that $E_{P^{\lambda}}[(\xi-E_{P^{\lambda}}[\xi|\mathcal{C}%
])^{2}]-\alpha(P^{\lambda})=\beta$. By (\ref{eq-unique-1}), we deduce that
$\hat{\eta}_{1}=\hat{\eta}_{2}$ $\mathbb{P}$-a.s..
\end{proof}

\subsubsection{Properties of the minimum mean square estimator}
Finally, in this subsection, we will list some properties of the MMSE $\rho(\xi|\mathcal{C})$.

\begin{proposition}
\label{prop-basic properties} If a convex operator $\rho(\cdot)$ is stable and
proper, then for any $\xi\in L_{\mathcal{F}}^{4p}(\mathbb{P})$, we have:

i) If $C_{1}\leq\xi(\omega)\leq C_{2}$ for two constants $C_{1}$ and $C_{2}$,
then $C_{1}\leq\rho(\xi|\mathcal{C})\leq C_{2}$;

ii) $\rho(-\xi|\mathcal{C})=-\rho(\xi|\mathcal{C})$;

iii) For any given $\eta_{0}\in L_{\mathcal{C}}^{2p}(\mathbb{P})$, we have
$\rho(\xi+\eta_{0}|\mathcal{C})=\rho(\xi|\mathcal{C})+\eta_{0}$;

iv) If $\xi$ is independent of the sub $\sigma$-algebra $\mathcal{C}$ under
every probability measure $P\in\mathcal{P}$, then $\rho(\xi|\mathcal{C})$ is a constant.
\end{proposition}

\begin{proof}
i) If $C_{1}\leq\xi(\omega)\leq C_{2}$, then for any $P\in\mathcal{P}$,
$C_{1}\leq E_{P}[\xi|\mathcal{C}]\leq C_{2}$. According to the proof of
Theorem \ref{unique}, $\rho(\xi|\mathcal{C})\in\{E_{P}[\xi|\mathcal{C}%
];P\in\mathcal{P}\}$ which leads to $C_{1}\leq\rho(\xi|\mathcal{C})\leq C_{2}$.

ii) Since
\[
\rho(\xi-\rho(\xi|\mathcal{C}))^{2}=\inf_{\eta\in L_{\mathcal{C}}%
^{2p}(\mathbb{P})}\rho(\xi-\eta)^{2}=\inf_{\eta\in L_{\mathcal{C}}%
^{2p}(\mathbb{P})}\rho(\xi+\eta)^{2}=\inf_{\eta\in L_{\mathcal{C}}%
^{2p}(\mathbb{P})}\rho(-\xi-\eta)^{2},
\]
we have
\[
\rho(-\xi-(-\rho(\xi|\mathcal{C})))^{2}=\inf_{\eta\in L_{\mathcal{C}}%
^{2p}(\mathbb{P})}\rho(-\xi-\eta)^{2}.
\]
By Theorem \ref{unique}, $-\rho(\xi|\mathcal{C})=\rho(-\xi|\mathcal{C})$.

iii) Note that
\[
\rho(\xi+\eta_{0}-(\eta_{0}+\rho(\xi|\mathcal{C})))^{2}=\rho(\xi-\rho
(\xi|\mathcal{C}))^{2}=\inf_{\eta\in L_{\mathcal{C}}^{2p}(\mathbb{P})}\rho
(\xi-\eta)^{2}=\inf_{\eta\in L_{\mathcal{C}}^{2p}(\mathbb{P})}\rho(\xi
+\eta_{0}-\eta)^{2}.
\]
By Theorem \ref{unique}, we have $\eta_{0}+\rho(\xi|\mathcal{C})=\rho(\xi
+\eta_{0}|\mathcal{C})$.

iv) If $\xi$ is independent of the sub $\sigma$-algebra $\mathcal{C}$ under
every $P\in\mathcal{P}$, then $E_{P}[\xi|\mathcal{C}]$ is a constant for any
$P\in\mathcal{P}$. Since $\rho(\xi|\mathcal{C})\in\{E_{P}[\xi|\mathcal{C}%
];P\in\mathcal{P}\}$, we know that $\rho(\xi|\mathcal{C})$ is a constant. This completes the proof.
\end{proof}

\section{Appendix}
%% if no title is needed, leave empty \section*{}.
For the convenience of the reader, we list the main theorems used in our proofs.

\begin{theorem}
[Fan \cite{Fan} Theorem 2]\label{minmax} Let $\mathcal{X}$ be a compact Hausdorff space and $\mathcal{Y}$ be an arbitrary set. Let $F$ be a real valued function defined on $\mathcal{X}\times \mathcal{Y}$ such that, for every $y\in\mathcal{Y}$, $F(x,y)$ is a $l.s.c$(lower-semicontinuous) on $\mathcal{X}$. If $F$ is convex on $\mathcal{X}$ and concave on $\mathcal{Y}$, then\\
$$\min_{x\in\mathcal{X}}\sup_{y\in\mathcal{Y}}F(x,y)=\sup_{y\in\mathcal{Y}}\min_{x\in\mathcal{X}}F(x,y).$$
\end{theorem}

\begin{proof}
Refer to Theorem 2 in \cite{Fan}.
\end{proof}

\begin{theorem}
[Z\v{a}linescu \cite{Zal} Theorem 2.10.1]\label{saddle point} Let $A$ and $B$ be
two nonempty sets and $f$ from $A\times B$ to $\mathbb{R}\bigcup\{\infty\}$.
Then $f$ has saddle points, i.e., there exists $(\bar{x},\bar{y})\in A\times
B$ such that
\[
\forall x\in A,\,\forall y\in B:\quad f(x,\bar{y})\leq f(\bar{x},\bar{y})\leq
f(\bar{x},y)
\]
if and only if
\[
\inf_{y\in B}f(\bar{x},y)=\max_{x\in A}\inf_{y\in B}f(x,y)=\min_{y\in B}%
\sup_{x\in A}f(x,y)=\sup_{x\in A}f(x,\bar{y}).
\]

\end{theorem}

\begin{theorem}
[Girsanov \cite{Girsanov}]\label{convex AP} We suppose that $\phi(t,\omega)$
satisfies the following conditions:\newline(1)\ $\phi(\cdot,\cdot)$ are
measurable in both variables;\newline(2)\ $\phi(t,\cdot)$ is $\mathcal{F}_{t}%
$-measurable for fixed $t$;\newline(3)\ $\int_{0}^{T}|\phi(t,\omega
)|^{2}dt<\infty$ almost everywhere; and $0<c_{1}\leq|\phi(t,\omega)|\leq
c_{2}$ for almost all $(t,\omega)$, then $\exp[\alpha\zeta_{s}^{t}(\phi)]$ is
integrable and for $\alpha>1$
\begin{equation}
\exp\big[\frac{(\alpha^{2}-\alpha)}{2}(t-s)c_{1}^{2}\big]\leq\mathbb{E}%
[\exp[\alpha\zeta_{s}^{t}(\phi)]]\leq\exp\big[\frac{(\alpha^{2}-\alpha)}%
{2}(t-s)c_{2}^{2}\big] \label{conclusion}%
\end{equation}
where $\zeta_{s}^{t}(\phi)=\int_{s}^{t}\phi(u,\omega)dw_{u}-\frac{1}{2}%
\int_{s}^{t}\phi^{2}(u,\omega)du$.
\end{theorem}

\begin{theorem}
[K$\hat{o}$saku Yosida \cite{Yosida}]\label{Mazur} Let $(X,\Vert\cdot\Vert)$
be a Banach space and $\{x_{n}\}_{n\in\mathbb{N}}$ be a sequence in $X$ that
converges weakly to some $x\in X$. Then there exists, for any $\epsilon>0$, a
convex combination $\sum_{j=1}^{n}\alpha_{j}x_{j},\ (\alpha_{j}\geq
0,\ \sum_{j=1}^{n}\alpha_{j}=1)$ such that $\Vert x-\sum_{j=1}^{n}\alpha
_{j}x_{j}\Vert\leq\epsilon$.
\end{theorem}

\begin{lemma}\label{calculate}
Let $\hat{\eta}_1=E_{\hat{P}_1}[\xi|\mathcal{C}]$, $\hat{\eta}_2=E_{\hat{P}_2}[\xi|\mathcal{C}]$, $P^{\lambda}=\lambda\hat{P}_1+(1-\lambda)\hat{P}_2$, $\lambda_{\hat{P}_1}=\lambda E_{P^{\lambda}}\big[\frac{d\hat{P}_1}{dP^{\lambda}}|\mathcal{C}\big]$, $\lambda_{\hat{P}_2}=(1-\lambda) E_{P^{\lambda}}\big[\frac{d\hat{P}_2}{dP^{\lambda}}|\mathcal{C}\big]$. Then we have
\begin{equation*}
\begin{aligned}
&E_{P^{\lambda}}[(\xi-\lambda_{\hat{P}_1}\hat{\eta}_1-\lambda_{\hat{P}_2}\hat{\eta}_2)^2]-\alpha(P^{\lambda})\\
=&\lambda E_{\hat{P}_1}\big[(\xi-\hat{\eta}_1)^2\big]+(1-\lambda) E_{\hat{P}_2}\big[(\xi-\hat{\eta}_2)^2\big]\\
&+\lambda E_{\hat{P}_1}\big[\lambda^2_{\hat{P}_2}(\hat{\eta}_1-\hat{\eta}_2)^2\big]+(1-\lambda)E_{\hat{P}_2}\big[\lambda^2_{\hat{P}_1}(\hat{\eta}_1-\hat{\eta}_2)^2\big]-\alpha(P^{\lambda}).
\end{aligned}
\end{equation*}
\end{lemma}

\begin{proof}
\begin{equation}\label{A1}
\begin{aligned}
&E_{P^{\lambda}}[(\xi-\lambda_{\hat{P}_1}\hat{\eta}_1-\lambda_{\hat{P}_2}\hat{\eta}_2)^2]-\alpha(P^{\lambda})\\
=&E_{P^{\lambda}}\big[\big(\lambda_{\hat{P}_1}(\xi-\hat{\eta}_1)+\lambda_{\hat{P}_2}(\xi-\hat{\eta}_2)\big)^2\big]-\alpha(P^{\lambda})\\
=&E_{P^{\lambda}}\big[\lambda^2_{\hat{P}_1}(\xi-\hat{\eta}_1)^2+\lambda^2_{\hat{P}_2}(\xi-\hat{\eta}_2)^2+2\lambda_{\hat{P}_1}\lambda_{\hat{P}_2}(\xi-\hat{\eta}_1)(\xi-\hat{\eta}_1)\big]-\alpha(P^{\lambda})\\
=&E_{P^{\lambda}}\big[\lambda_{\hat{P}_1}(\xi-\hat{\eta}_1)^2+\lambda_{\hat{P}_2}(\xi-\hat{\eta}_2)^2-\lambda_{\hat{P}_1}\lambda_{\hat{P}_2}(\hat{\eta}_1-\hat{\eta}_2)^2\big]-\alpha(P^{\lambda})\\
=&\lambda E_{\hat{P}_1}[\lambda_{\hat{P}_1}(\xi-\hat{\eta}_1)^2]+(1-\lambda)E_{\hat{P}_2}[\lambda_{\hat{P}_1}(\xi-\hat{\eta}_1)^2]+\lambda E_{\hat{P}_1}[\lambda_{\hat{P}_2}(\xi-\hat{\eta}_2)^2]
\\&+(1-\lambda)E_{\hat{P}_2}[\lambda_{\hat{P}_2}(\xi-\hat{\eta}_2)^2]-\lambda E_{\hat{P}_1}[\lambda_{\hat{P}_1}\lambda_{\hat{P}_2}(\hat{\eta}_1-\hat{\eta}_2)^2]\\
&-(1-\lambda)E_{\hat{P}_2}[\lambda_{\hat{P}_1}\lambda_{\hat{P}_2}(\hat{\eta}_1-\hat{\eta}_2)^2]-\alpha(P^{\lambda})\\
=&\lambda E_{\hat{P}_1}[(\xi-\hat{\eta}_1)^2]+(1-\lambda)E_{\hat{P}_2}[(\xi-\hat{\eta}_1)^2]-(1-\lambda)E_{\hat{P}_2}[\lambda_{\hat{P}_2}(\xi-\hat{\eta}_1)^2]\\
&+\lambda E_{\hat{P}_1}[(\xi-\hat{\eta}_2)^2]-\lambda E_{\hat{P}_1}[\lambda_{\hat{P}_1}(\xi-\hat{\eta}_2)^2]+(1-\lambda) E_{\hat{P}_2}[(\xi-\hat{\eta}_2)^2]\\&-\lambda E_{\hat{P}_1}[\lambda_{\hat{P}_2}(\xi-\hat{\eta}_1)^2]-(1-\lambda) E_{\hat{P}_2}[\lambda_{\hat{P}_1}(\xi-\hat{\eta}_2)^2]-\lambda E_{\hat{P}_1}[\lambda_{\hat{P}_1}\lambda_{\hat{P}_2}(\hat{\eta}_1-\hat{\eta}_2)^2]\\
&-(1-\lambda)E_{\hat{P}_2}[\lambda_{\hat{P}_1}\lambda_{\hat{P}_2}(\hat{\eta}_1-\hat{\eta}_2)^2]-\alpha(P^{\lambda}).
\end{aligned}
\end{equation}

Since
$$(1-\lambda)E_{\hat{P}_2}[(\xi-\hat{\eta}_1)^2]=(1-\lambda)E_{\hat{P}_2}[(\lambda_{\hat{P}_1}+\lambda_{\hat{P}_2})(\xi-\hat{\eta}_1)^2]$$
and
$$\lambda E_{\hat{P}_1}[(\xi-\hat{\eta}_2)^2]=\lambda  E_{\hat{P}_1}[(\lambda_{\hat{P}_1}+\lambda_{\hat{P}_2})(\xi-\hat{\eta}_2)^2],$$
it results in that
\begin{equation*}
\begin{aligned}
\eqref{A1}&=\lambda E_{\hat{P}_1}[\lambda_{\hat{P}_2}(\xi-\hat{\eta}_2)^2-\lambda_{\hat{P}_2}(\xi-\hat{\eta}_1)^2]+(1-\lambda)E_{\hat{P}_2}[\lambda_{\hat{P}_1}(\xi-\hat{\eta}_1)^2-\lambda_{\hat{P}_1}(\xi-\hat{\eta}_2)^2]\\
&-\lambda E_{\hat{P}_1}[\lambda_{\hat{P}_1}\lambda_{\hat{P}_2}(\hat{\eta}_1-\hat{\eta}_2)^2]-(1-\lambda)E_{\hat{P}_2}[\lambda_{\hat{P}_1}\lambda_{\hat{P}_2}(\hat{\eta}_1-\hat{\eta}_2)^2]+\lambda E_{\hat{P}_1}[(\xi-\hat{\eta}_1)^2]\\
&+(1-\lambda)E_{\hat{P}_2}[(\xi-\hat{\eta}_2)^2].
\end{aligned}
\end{equation*}
Firstly, we calculate the items with respect to the expectation $\lambda E_{\hat{P}_1}[\cdot]$, the following relations hold:
\begin{equation*}
\begin{aligned}
&\lambda_{\hat{P}_2}(\xi^2+\hat{\eta}_2^2-2\xi\hat{\eta}_2)-\lambda_{\hat{P}_2}(\xi^2+\hat{\eta}_1^2-2\xi\hat{\eta}_1)-\lambda_{\hat{P}_1}\lambda_{\hat{P}_2}(\hat{\eta}_1-\hat{\eta}_2)^2\\
=&\lambda_{\hat{P}_2}[2\hat{\eta}_1(\hat{\eta}_2-\hat{\eta}_1)+2\xi(\hat{\eta}_1-\hat{\eta}_2)]+\lambda_{\hat{P}_2}^2(\hat{\eta}_1-\hat{\eta}_2)^2\\
=&\lambda_{\hat{P}_2}[2(\xi-\hat{\eta}_1)(\hat{\eta}_1-\hat{\eta}_2)]+\lambda_{\hat{P}_2}^2(\hat{\eta}_1-\hat{\eta}_2)^2.
\end{aligned}
\end{equation*}

Since $\lambda_{\hat{P}_2}(\hat{\eta}_1-\hat{\eta}_2)$ is $\mathcal{C}$-measurable and $(\xi-\hat{\eta}_1)$ is orthogonal with $\sigma$-algebra $\mathcal{C}$ under probability measure $\hat{P}_1$, it results that
$$\lambda E_{\hat{P}_1}[\lambda_{\hat{P}_2}2(\xi-\hat{\eta}_1)(\hat{\eta}_1-\hat{\eta}_2)]=\lambda E_{\hat{P}_1}[\lambda_{\hat{P}_2}(\hat{\eta}_1-\hat{\eta}_2)]E_{\hat{P}_1}[2(\xi-\hat{\eta}_1)]=0.$$

Secondly, we can also similarly calculate the items with respect to the expectation $(1-\lambda) E_{\hat{P}_2}[\cdot]$. Finally, the equation \eqref{A1} can be expressed as

\begin{equation*}
\begin{aligned}
&E_{P^{\lambda}}[(\xi-\lambda_{\hat{P}_1}\hat{\eta}_1-\lambda_{\hat{P}_2}\hat{\eta}_2)^2]-\alpha(P^{\lambda})\\
=&\lambda E_{\hat{P}_1}\big[(\xi-\hat{\eta}_1)^2\big]+(1-\lambda) E_{\hat{P}_2}\big[(\xi-\hat{\eta}_2)^2\big]\\
&+\lambda E_{\hat{P}_1}\big[\lambda^2_{\hat{P}_2}(\hat{\eta}_1-\hat{\eta}_2)^2\big]+(1-\lambda)E_{\hat{P}_2}\big[\lambda^2_{\hat{P}_1}(\hat{\eta}_1-\hat{\eta}_2)^2\big]-\alpha(P^{\lambda}).
\end{aligned}
\end{equation*}
This completes the proof.
\end{proof}


\begin{thebibliography}{4}
%%
%\bibitem{r1}
%\textsc{Billingsley, P.} (1999). \textit{Convergence of
%Probability Measures}, 2nd ed.
%Wiley, New York.
%
%\bibitem{r2}
%\textsc{Bourbaki, N.}  (1966). \textit{General Topology}  \textbf{1}.
%Addison--Wesley, Reading, MA.
%
%\bibitem{r3}
%\textsc{Ethier, S. N.} and \textsc{Kurtz, T. G.} (1985).
%\textit{Markov Processes: Characterization and Convergence}.
%Wiley, New York.
%
%\bibitem{r4}
%\textsc{Prokhorov, Yu.} (1956).
%Convergence of random processes and limit theorems in probability
%theory. \textit{Theory  Probab.  Appl.}
%\textbf{1} 157--214.
%%%%%%%%%%%%%%%%%%%%%%%%%%%%%%%%%%%%%%%%%%%%%%%%%%%%%%%%%%%%%%%%%%%%%%%%%%%%%%%%%%%%%%
%%%%%%%%%%%%%%%%%%%%%%%%%%%%%%%%%%%%%%%%%%%%%%%%%%%%%%%%%%%%%%%%%%%%%%%%%%%%%%%%%%%%%%
%%%%%%%%%%%%%%%%%%%%%%%%%%%%%%%%%%%%%%%%%%%%%%%%%%%%%%%%%%%%%%%%%%%%%%%%%%%%%%%%%%%%%%
%%%%%%%%%%%%%%%%%%%%%%%%%%%%%%%%%%%%%%%%%%%%%%%%%%%%%%%%%%%%%%%%%%%%%%%%%%%%%%%%%%%%%%
%%%%%%%%%%%%%%%%%%%%%%%%%%%%%%%%%%%%%%%%%%%%%%%%%%%%%%%%%%%%%%%%%%%%%%%%%%%%%%%%%%%%%%
\bibitem {Allan-Cohen} \textsc{Allan, A.} and \textsc{Cohen, S.} (2018). Parameter Uncertainty in the
Kalman-Bucy Filter. \textit{SIAM Journal on Control and Optimization}.

\bibitem {Arai-Fukasawa} \textsc{Arai, T.} and \textsc{Fukasawa, M.} (2014). Convex risk measure for good
deal bounds. \textit{Mathematical Finance}. 464-484.

\bibitem {Bensoussan}\textsc{Bensoussan, A.} (2004). \textit{Stochastic control of partially observable
systems}. Cambridge University Press.

\bibitem {Bensoussan-Keppo}\textsc{Bensoussan, A.} and \textsc{Keppo, J.} (2009). Optimal consumption and
portfolio decisions with partially observed real prices. \textit{Mathematical Finance}. \textbf{19}  215-236.

\bibitem {BC}\textsc{Bain, A.} and \textsc{Crisan, D.} (2009). \textit{Fundamentals of stochastic filtering}.
Springer Science and Business Media.

\bibitem {Borisov1}\textsc{Borisov, A. V.} (2008). Minimax a posteriori estimation of the
Markov processes with finite state spaces. \textit{Automation and Remote Control}. \textbf{69} 233-246.

\bibitem {Borisov2}\textsc{Borisov, A. V.} (2011). The Wonham filter under uncertainty: A
game-theoretic approach. \textit{Automatica}. \textbf{47} 1015-1019.

\bibitem {Chen-Epstein}\textsc{Chen, Z} and \textsc{Epstein, L.} (2002). Ambiguity, risk, and asset
returns in continuous time. \textit{Econometrica}. \textbf{70} 1403-1443.

\bibitem {Chen-Shen} \textsc{Chen, G} and \textsc{Shen, Y.} (2009). Robust $H_{\infty}$ filter design for neutral stochastic uncertain systems
with time-varying delay. \textit{Journal of Mathematical Analysis and Applications}. \textbf{353} 196-204.

\bibitem {Che-Yang} \textsc{Che, W} and \textsc{Yang, G.} (2013). $H_{\infty}$ filter design for continuous-time systems with quantised signals. \textit{International Journal of Systems Science}. \textbf{44} 265-274.

\bibitem {Delbaen2010}\textsc{Delbaen, F.} \textsc{Peng, S} and \textsc{Gianin, E. R.} (2010) Representation of
the penalty term of dynamic concave utilities. \textit{Finance Stoch}. \textbf{14} 449472.

\bibitem {Duncan1}\textsc{Duncan,T. E.} and \textsc{Pasik-Duncan, B.} (2014). Some Results on Optimal
Control for a Partially Observed Linear Stochastic System with an Exponential
Quadratic Cost. \textit{IFAC Proceedings}. \textbf{47}.

\bibitem {EPQ}\textsc{El Karoui, N.} \textsc{Peng, S.} and \textsc{Quenez, M} (1997). Backward stochastic
differential equations in finance. \textit{Math. Finance}. \textbf{7} 1-71.

\bibitem {Epstein-Ji-1}\textsc{Epstein, L.} and \textsc{Ji, S.} (2013). Ambiguous Volatility, Possibility
and Utility in Continuous Time. \textit{Journal of Mathematical Economics}. 269-282.

\bibitem {Epstein-Ji-2}\textsc{Epstein, L.} and \textsc{Ji, S.} (2013). Ambiguous volatility and asset
pricing in continuous time. \textit{Rev. Finan. Stud.} 1740-1786.

\bibitem{Fan} \textsc{Fan, K.} (1953). Minimax theorems. \textit{Proceedings of the National Academy of Sciences of U.S.A}. \textbf{39} 42-47.

\bibitem {FS}\textsc{F\"{o}llmer, H.} and \textsc{Schied, A.} (2002). \textit{Stochastic Finance, An introduction
in discrete time}. Walter de Gruyter, Berlin/New York.

\bibitem {Girsanov} \textsc{Girsanov, I. V.} (1960). On transforming a certain class of
stochastic processes by absolutely continuous substitution of measures. \textit{Theory
of Probability and Its Applications}. \textbf{5} 285-301.

\bibitem{Guo} \textsc{Guo, L.} (2020). Estimation, control, and games of dynamical systems with uncertainty. \textit{SCIENCE CHINA Information Sciences}. \textbf{50}.

\bibitem{Huang-Wang-Zhang} \textsc{Huang, P.} \textsc{Wang, G.} and \textsc{Zhang, H.} (2020). A partial information linear-quadratic optimal control problem of backward stochastic differential equation with its applications. \textit{SCIENCE CHINA Information Sciences}. \textbf{63}.

\bibitem {Ji-Kong-Sun}\textsc{Ji, S.}  \textsc{Kong, C.} and \textsc{Sun, C.} (2020). A filtering problem with
uncertainty in observation. \textit{System and Control Letters}. 1-5.

\bibitem {Ji-Kong-Sun-amtomatica}\textsc{Ji, S.} \textsc{Kong, C.} and \textsc{Sun, C.} (2020). A robust
Kalman-Bucy filtering problem. \textit{Automatica}. https://doi.org/10.1016/j.automatica.2020.109252.

\bibitem {Ji-Kong-Sun-1}\textsc{Ji, S.}  \textsc{Kong, C.} and \textsc{Sun, C.} (2019). The minimum mean square
estimator of integrable variables under sublinear operators. \textit{Stochastics}. 519-532.

\bibitem {KR} \textsc{Kaina, M} and \textsc{R\"{u}schendorf, L.} (2009). On convex risk measures on
$L_{p}$-spaces. \textit{Math. Meth. Oper. Res}. \textbf{69} (2009) 475-495.

\bibitem {K-S}\textsc{Karatzas, I.} and \textsc{Shreve, S.} (2002). \textit{Brownian Motion and Stochastic
Calculus}. Springer-Verlag.

\bibitem {L} \textsc{Lakner, P.}  (1995). Utility maximization with partial information.
\textit{Stochastic processes and their applications}. 247-273.

\bibitem {Liptser}\textsc{Liptser, R. S.} and \textsc{Shiryaev, A. N.} (2013). \textit{Statistics of random
Processes: I. General Theory}. Springer Science and Business Media.

\bibitem {Oksendal-Sulem} \textsc{{\O }ksendal, B.} and \textsc{Sulem, A.} (2014). Forward-Backward
Stochastic Differential Games and Stochastic Control under Model Uncertainty.
\textit{Journal of Optimization Theory and Applications}. 22-55.

\bibitem {Peng-1} \textsc{Peng, S.} (1997). BSDE and related g-expectations, Backward Stochastic
Differential Equations, El Karoui, N. and Mazliak, L. eds., Pitman Research
Notes in Mathematics Series, 364: 141--159, Longman, Harlow.

\bibitem {Pham}\textsc{Pham, H.} (2009). \textit{Continuous-time stochastic control and optimization
with financial applications}. Springer Science and Business Media.

\bibitem {Simons}\textsc{Simons, S.} (2008). \textit{From Hahn-Banach to Monotonicity}. Springer.

\bibitem {JS}\textsc{Sun, C.} and \textsc{Ji, S.} (2017) The least squares estimator of random variables
under sublinear expectations. \textit{Journal of Mathematical Analysis and
Applications}. \textbf{451} 906-923.

\bibitem {Tang} \textsc{Tang, S.} (1998). The maximum principle for partially observed optimal
control of stochastic differential equations. \textit{SIAM J.Control Optim}. 1956-1617.

\bibitem {Xiong}\textsc{Xiong, J.} (2008). \textit{An introduction to stochastic filtering theory}.
Oxford University Press.

\bibitem {Yong-Zhou}\textsc{Yong, J.} and \textsc{Zhou, X.} (1999). \textit{Stochastic controls: Hamiltonian
systems and HJB equations}. Springer Science and Business Media.

\bibitem {Yosida} \textsc{K$.$ Yosida} (1980). \textit{Functional Analysis}. Springer Berlin Heidelberg.

\bibitem {Zal}\textsc{Z\v{a}linescu, C.} (2002). \textit{Convex Analysis in General Vector Spaces}.
World Scientific, River Edge, NJ.
\end{thebibliography}
\end{document}